\documentclass[reqno]{amsart}
\pdfoutput=1
\usepackage{amsmath}
\usepackage{amsfonts}
\usepackage{amssymb}
\usepackage{graphicx}%
\providecommand{\U}[1]{\protect\rule{.1in}{.1in}}
\newtheorem{theorem}{Theorem}

\newtheorem{conclusion}[theorem]{Conclusion}

\newtheorem{corollary}[theorem]{Corollary}

\newtheorem{definition}[theorem]{Definition}

\newtheorem{goal}[theorem]{Goal}
\newtheorem{lemma}[theorem]{Lemma}

\newtheorem{proposition}[theorem]{Proposition}

\begin{document}

\title{PDE methods in random matrix theory}
\author{Brian C. Hall}
\email{bhall@nd.edu}
\address{Department of Mathematics, University of Notre Dame, Notre Dame,
IN 46556, USA}
\thanks{Supported in part by a Simons Foundation Collaboration Grant for Mathematicians}

\maketitle

\begin{abstract}
This article begins with a brief review of random matrix theory, followed by a
discussion of how the large-$N$ limit of random matrix models can be realized
using operator algebras. I then explain the notion of \textquotedblleft Brown
measure,\textquotedblright\ which play the role of the eigenvalue distribution
for operators in an operator algebra.

I then show how methods of partial differential equations can be used to
compute Brown measures. I consider in detail the case of the circular law and
then discuss more briefly the case of the free multiplicative Brownian motion,
which was worked out recently by the author with Driver and Kemp.

\end{abstract}
\tableofcontents

\setcounter{tocdepth}{2}

\section{Random matrices}

Random matrix theory consists of choosing an $N\times N$ matrix at random and
looking at natural properties of that matrix, notably its eigenvalues.
Typically, interesting results are obtained only for \textit{large} random
matrices, that is, in the limit as $N$ tends to infinity. The subject began
with the work of Wigner \cite{Wigner}, who was studying energy levels in large
atomic nuclei. The subject took on new life with the discovery that the
eigenvalues of certain types of large random matrices resemble the energy
levels of quantum chaotic systems---that is, quantum mechanical systems for
which the underlying classical system is chaotic. (See, for example,
\cite{Gutz} or \cite{Stockmann}.) There is also a fascinating conjectural
agreement, due to Montgomery \cite{Mon}, between the statistical behavior of
zeros of the Riemann zeta function and the eigenvalues of random matrices. See
also \cite{KS} or \cite{BK}.

We will review briefly some standard results in the subject, which may be
found in textbooks such as those by Tao \cite{Tao} or Mehta \cite{Mehta}.

\subsection{The Gaussian unitary ensemble\label{GUE.sec}}

The first example of a random matrix is the \textbf{Gaussian unitary ensemble}
(GUE) introduced by Wigner \cite{Wigner}. Let $H_{N}$ denote the real vector
space of $N\times N$ Hermitian matrices, that is, those with $X^{\ast}=X,$
where $X^{\ast}$ is the conjugate transpose of $X.$ We then consider a
Gaussian measure on $H_{N}$ given by%
\begin{equation}
d_{N}e^{-N\mathrm{trace}(X^{2})/2}~dX,\quad X\in H_{N}, \label{GUEdensity}%
\end{equation}
where $dX$ denotes the Lebesgue measure on $H_{N}$ and where $d_{N}$ is a
normalizing constant. If $X^{N}$ is a random matrix having this measure as its
distribution, then the diagonal entries are normally distributed real random
variables with mean zero and variance $1/N.$ The off-diagonal entries are
normally distributed complex random variables, again with mean zero and
variance $1/N.$ Finally, the entries are as independent as possible given that
they are constrained to be Hermitian, meaning that the entries on and above
the diagonal are independent (and then the entries below the diagonal are
determined by those above the diagonal). The factor of $N$ in the exponent in
(\ref{GUEdensity}) is responsible for making the variance of the entries of
order $1/N.$ This scaling of the variances, in turn, guarantees that the
eigenvalues of the random matrix $X^{N}$ do not blow up as $N$ tends to infinity.

In order to state the first main result of random matrix theory, we introduce
the following notation.

\begin{definition}
\label{eed.def}For any $N\times N$ matrix $X,$ the \textbf{empirical
eigenvalue distribution} of $X$ is the probability measure on $\mathbb{C}$
given by%
\[
\frac{1}{N}\sum_{j=1}^{N}\lambda_{j},
\]
where $\{\lambda_{1},\ldots,\lambda_{N}\}$ are the eigenvalues of $X$, listed
with their algebraic multiplicity.
\end{definition}

We now state \textbf{Wigner's semicircle law}.%

\begin{figure}[ptb]%
\centering
\includegraphics[
height=1.6535in,
width=2.5278in
]%
{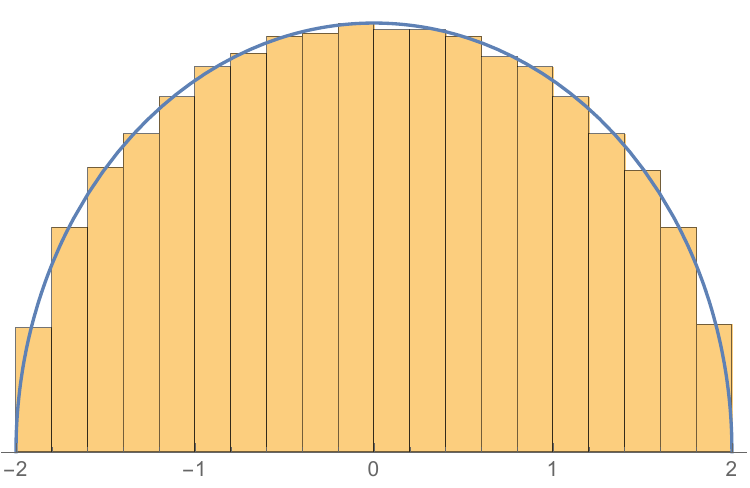}%
\caption{A histogram of the eigenvalues of a GUE random variable with
$N=2,000,$ plotted against a semicircular density}%
\label{guehist.fig}%
\end{figure}

\begin{theorem}
\label{wigner.thm}Let $X^{N}$ be a sequence of independently chosen $N\times
N$ random matrices, each chosen according to the probability distribution in
(\ref{GUEdensity}). Then as $N\rightarrow\infty,$ the empirical eigenvalue
distribution of $X^{N}$ converges almost surely in the weak topology to
Wigner's semicircle law, namely the measure supported on $[-2,2]$ and given
there by%
\begin{equation}
\frac{1}{2\pi}\sqrt{4-x^{2}}~dx,\quad-2\leq x\leq2. \label{semi}%
\end{equation}

\end{theorem}

Figure \ref{guehist.fig} shows a simulation of the Gaussian unitary ensemble
for $N=2,000,$ plotted against the semicircular density in (\ref{semi}). One
notable aspect of Theorem \ref{wigner.thm} is that the limiting eigenvalue
distribution (i.e., the semicircular measure in (\ref{semi})) is
\textit{nonrandom}. That is to say, we are choosing a matrix at random, so
that its eigenvalues are random, but in the large-$N$ limit, the randomness in
the bulk eigenvalue distribution disappears---it is always semicircular. Thus,
if we were to select another GUE matrix with $N=2,000$ and plot its
eigenvalues, the histogram would (with high probability) look very much like
the one in Figure \ref{guehist.fig}.

It is important to note, however, that if one zooms in with a magnifying glass
so that one can see the individual eigenvalues of a large GUE matrix, the
randomness in the eigenvalues will persist. The behavior of these individual
eigenvalues is of considerable interest, because they are supposed to resemble
the energy levels of a \textquotedblleft quantum chaotic
system\textquotedblright\ (that is, a quantum mechanical system whose
classical counterpart is chaotic). Nevertheless, in this article, I will deal
only with the bulk properties of the eigenvalues.

\subsection{The Ginibre ensemble\label{ginibre.sec}}%

\begin{figure}[ptb]%
\centering
\includegraphics[
height=2.5278in,
width=2.5278in
]%
{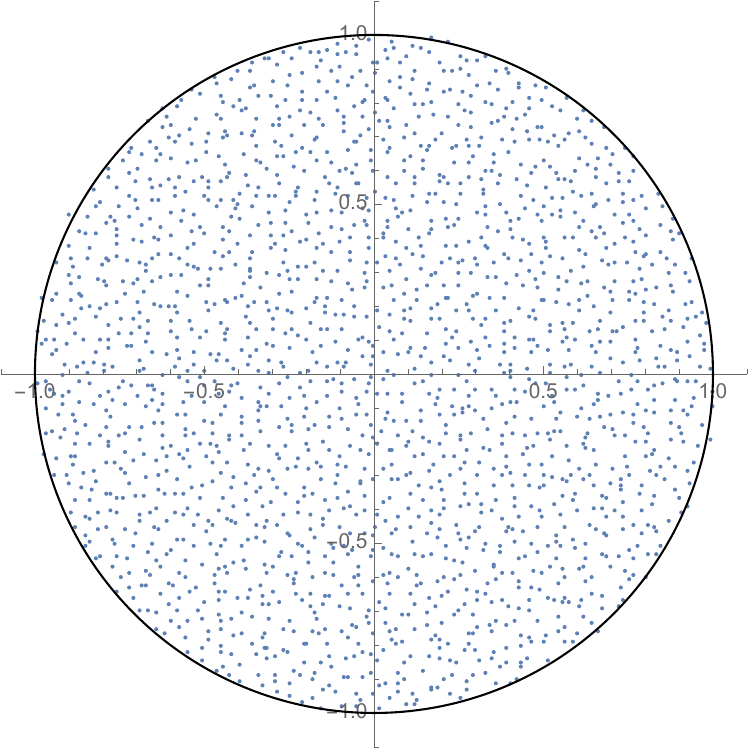}%
\caption{A plot of the eigenvalues of a Ginibre matrix with $N=2,000$}%
\label{ginibreplot.fig}%
\end{figure}

We now discuss the non-Hermitian counterpart to the Gaussian unitary ensemble,
known as the \textbf{Ginibre ensemble} \cite{Gin}.We let $M_{N}(\mathbb{C})$
denote the space of \textit{all} $N\times N$ matrices, not necessarily
Hermitian. We then make a measure on $M_{N}(\mathbb{C})$ using a formula
similar to the Hermitian case:%
\begin{equation}
f_{N}~e^{-N\mathrm{trace}(Z^{\ast}Z)}~dZ,\quad Z\in M_{N}(\mathbb{C}),
\label{GinibreDensity}%
\end{equation}
where $dZ$ denotes the Lebesgue measure on $H_{N}$ and where $f_{N}$ is a
normalizing constant. In this case, \textit{all} the entries of $Z$ are
independent of one another. Each entry is a complex-valued normal random
variable with mean zero and variance $1/N.$

The eigenvalues for the Ginibre ensemble need not be real and they follow the
\textbf{circular law}.

\begin{theorem}
\label{circular.thm}Let $Z^{N}$ be a sequence of independently chosen $N\times
N$ random matrices, each chosen according to the probability distribution in
(\ref{GinibreDensity}). Then as $N\rightarrow\infty,$ the empirical eigenvalue
distribution of $Z^{N}$ converges almost surely in the weak topology to the
uniform measure on the unit disk.
\end{theorem}

Figure \ref{ginibreplot.fig} shows the eigenvalues of a random matrix chosen
from the Ginibre ensemble with $N=2,000.$ As in the GUE case, the bulk
eigenvalue distribution becomes deterministic in the large-$N$ limit. As in
the GUE case, one can also zoom in with a magnifying glass on the eigenvalues
of a Ginibre matrix until the individual eigenvalues become visible, and the
local behavior of these eigenvalues is an interesting problem---which will not
be discussed in this article.

\subsection{The Ginibre Brownian motion\label{ginBrownian.sec}}

In this article, I will discuss a certain approach to analyzing the behavior
of the eigenvalues in the Ginibre ensemble. The main purpose of this analysis
is not so much to obtain the circular law, which can be proved by various
other methods. The main purpose is rather to develop tools that can be used to
study a more complex random matrix model in the group of \textit{invertible}
$N\times N$ matrices. The Ginibre case then represents a useful prototype for
this more complicated problem.

It is then useful to introduce a time-parameter into the description of the
Ginibre ensemble, which we can do by studying the \textbf{Ginibre Brownian
motion}. Specifically, in any finite-dimensional real inner product space $V$,
there is a natural notion of Brownian motion. The Ginibre Brownian motion is
obtained by taking $V$ to be $M_{N}(\mathbb{C})$, viewed as a real vector
space of dimension $2N^{2},$ and using the (real) inner product $\left\langle
\cdot,\cdot\right\rangle _{N}$ given by%
\[
\left\langle X,Y\right\rangle _{N}:=N\operatorname{Re}(\mathrm{trace}(X^{\ast
}Y)).
\]
We let $C_{t}^{N}$ denote this Brownian motion, assumed to start at the origin.

At any one fixed time, the distribution of $C_{t}^{N}$ is just the same as
$\sqrt{t}Z^{N},$ where $Z^{N}$ is distributed as the Ginibre ensemble. The
\textit{joint} distribution of the process $C_{t}^{N}$ for various values of
$t$ is determined by the following property: For any collection of times
$0=t_{0}<t_{1}<t_{2}<\cdots<t_{k},$ the \textquotedblleft
increments\textquotedblright%
\begin{equation}
C_{t_{1}}^{N}-C_{t_{0}}^{N},C_{t_{2}}^{N}-C_{t_{1}}^{N},\ldots,C_{t_{k}}%
^{N}-C_{t_{k-1}}^{N} \label{ginIncrements}%
\end{equation}
are independent and distributed as $\sqrt{t_{j}-t_{j-1}}Z^{N}.$

\section{Large-$N$ limits in random matrix theory}

Results in random matrix theory are typically expressed by first computing
some quantity (e.g., the empirical eigenvalue distribution) associated to an
$N\times N$ random matrix and then letting $N$ tend to infinity. It is
nevertheless interesting to ask whether there is some sort of limiting object
that captures the large-$N$ limit of the entire random matrix model. In this
section, we discuss one common approach constructing such a limiting object.

\subsection{Limit in $\ast$-distribution\label{StarLim.sec}}

Suppose we have a matrix-valued random variable $X,$ not necessarily normal.
Then we can then speak about the $\ast$-moments of $X,$ which are expressions
like
\[
\mathbb{E}\left\{  \frac{1}{N}\mathrm{trace}(X^{2}(X^{\ast})^{3}X^{4}X^{\ast
})\right\}  .
\]
Generally, suppose $p(a,b)$ is a polynomial in two noncommuting variables,
that is, a linear combination of words involving products of $a$'s and $b$'s
in all possible orders. We may then consider%
\[
\mathbb{E}\left\{  \frac{1}{N}\mathrm{trace}[p(X,X^{\ast})]\right\}  .
\]
If, as usual, we have a \textit{family} $X^{N}$ of $N\times N$ random
matrices, we may consider the limits of such $\ast$-moments (if the limits
exist):%
\begin{equation}
\lim_{N\rightarrow\infty}\mathbb{E}\left\{  \frac{1}{N}\mathrm{trace}%
[p(X^{N},(X^{N})^{\ast})]\right\}  . \label{limitStar}%
\end{equation}

\subsection{Tracial von Neumann algebras\label{opAlg.sec}}

Our goal is now to find some sort of limiting object that can encode
\textit{all} of the limits in (\ref{limitStar}). Specifically, we will try to
find the following objects: (1) an operator algebra $\mathcal{A}$, (2) a
\textquotedblleft trace\textquotedblright\ $\tau:\mathcal{A}\rightarrow
\mathbb{C},$ and (3) and element $x$ of $\mathcal{A},$ such that for each
polynomial $p$ in two noncommuting variables, we have%
\begin{equation}
\lim_{N\rightarrow\infty}\mathbb{E}\left\{  \frac{1}{N}\mathrm{trace}%
[p(X^{N},(X^{N})^{\ast})]\right\}  =\tau\lbrack p(x,x^{\ast})].
\label{largeNlimit1}%
\end{equation}

We now explain in more detail what these objects should be. First, we
generally take $\mathcal{A}$ to be a von Neumann algebra, that is, an algebra
of operators that contains the identity, is closed under taking adjoints, and
is closed under taking weak operator limits. Second, the \textquotedblleft
trace\textquotedblright\ $\tau$ is not actually computed by taking the trace
of elements of $\mathcal{A},$ which are typically not of trace class. Rather,
$\tau$ is a linear functional that has properties similar to the properties of
the \textit{normalized} trace $\frac{1}{N}\mathrm{trace}(\cdot)$ for matrices.
Specifically, we require the following properties:

\begin{itemize}
\item $\tau(1)=1,$ where on the left-hand side, $1$ denotes the identity operator,

\item $\tau(a^{\ast}a)\geq0$ with equality only if $a=0,$ and

\item $\tau(ab)=\tau(ba),$ and

\item $\tau$ should be continuous with respect to the weak-$\ast$ topology on
$\mathcal{A}.$
\end{itemize}

\noindent Last, $x$ is a single element of $\mathcal{A}.$

We will refer to the pair $(\mathcal{A},\tau)$ as a \textbf{tracial von
Neumann algebra}. We will not discuss here the methods used for actually
constructing interesting examples of tracial von Neumann algebras. Instead, we
will simply accept as a known result that certain random matrix models admit
large-$N$ limits as operators in a tracial von Neumann algebra. (The
interested reader may consult the work of Biane and Speicher \cite{BS1}, who
use a Fock space construction to find tracial von Neumann algebras of the sort
we will be using in this article.)

Let me emphasize that although $X^{N}$ is a matrix-valued random variable, $x$
is \textit{not} an operator-valued random variable. Rather, $x$ is a
\textit{single} operator in the operator algebra $\mathcal{A}.$ This situation
reflects a typical property of random matrix models, which we have already
seen an example of in Sections \ref{GUE.sec} and \ref{ginibre.sec}, that
certain random quantities become nonrandom in the large-$N$ limit. In the
present context, it is often the case that we have a stronger statement
than\ (\ref{largeNlimit1}), as follows: If we sample the $X^{N}$'s
independently for different $N$'s, then with probability one, we will have%
\[
\lim_{N\rightarrow\infty}\frac{1}{N}\mathrm{trace}[p(X^{N},(X^{N})^{\ast
})]=\tau\lbrack p(x,x^{\ast})].
\]
That is to say, in many cases, the \textit{random} quantity $\frac{1}%
{N}\mathrm{trace}[p(X^{N},(X^{N})^{\ast})]$ converges almost surely to the
\textit{single, deterministic} number $\tau\lbrack p(x,x^{\ast})]$ as $N$
tends to infinity.

\subsection{Free independence\label{free.sec}}

In random matrix theory, it is often convenient to construct random matrices
as sums or products of other random matrices, which are frequently assumed to
be independent of one another. The appropriate notion of independence in the
large-$N$ limit---that is, in a tracial von Neumann algebra---is the notion of
\textquotedblleft freeness\textquotedblright\ or \textquotedblleft free
independence.\textquotedblright\ This concept was introduced by Voiculescu
\cite{Voi1,Voi2} and has become a powerful tool in random matrix theory. (See
also the monographs \cite{NS} by Nica and Speicher and \cite{MS} by Mingo and
Speicher.) Given an element $a$ in a tracial von Neumann algebra
$\mathcal{(A},\tau)$ and a polynomial $p,$ we may form the element $p(a).$ We
also let $\dot{p}(a)$ denote the corresponding \textquotedblleft
centered\textquotedblright\ element, given by%
\[
\dot{p}(a)=p(a)-\tau(p(a))
\]

We then say that elements $a_{1},\ldots,a_{k}$ are \textbf{freely independent}
(or, more concisely, \textbf{free}) if the following condition holds. Let
$j_{1},\ldots,j_{n}$ be any sequence of indices taken from $\{1,\ldots,k\}$,
with the property that $j_{l}$ is distinct from $j_{l+1}.$ Let $p_{j_{1}%
},\ldots,p_{j_{n}}$ be any sequence $p_{j_{1}},\ldots,p_{j_{n}}$ of
polynomials. Then we should have%
\[
\tau(\dot{p}_{j_{1}}(a_{j_{1}})\dot{p}_{j_{2}}(a_{j_{2}})\cdots\dot{p}_{j_{n}%
}(a_{j_{n}}))=0.
\]
Thus, for example, if $a$ and $b$ are freely independent, then%
\[
\tau\lbrack(a^{2}-\tau(a^{2}))(b^{2}-\tau(b^{2}))(a-\tau(a))]=0.
\]

The concept of freeness allows us, in principle, to disentangle traces of
arbitrary words in freely independent elements, thereby reducing the
computation to the traces of powers of individual elements. As an example, let
us do a few computations with two freely independent elements $a$ and $b$. We
form the corresponding centered elements $a-\tau(a)$ and $b-\tau(b)$ and start
applying the definition:%
\begin{align*}
0  &  =\tau\lbrack(a-\tau(a))(b-\tau(b))]\\
&  =\tau\lbrack ab]-\tau\lbrack\tau(a)b]-\tau\lbrack a\tau(b)]+\tau\lbrack
\tau(a)\tau(b)]\\
&  =\tau\lbrack ab]-\tau(a)\tau(b)-\tau(a)\tau(b)+\tau(a)\tau(b)\\
&  =\tau\lbrack ab]-\tau(a)\tau(b),
\end{align*}
where we have used that scalars can be pulled outside the trace and that
$\tau(1)=1.$ We conclude, then, that%
\[
\tau(ab)=\tau(a)\tau(b).
\]
A similar computation shows that $\tau(a^{2}b)=\tau(a^{2})\tau(b)$ and that
$\tau(ab^{2})=\tau(a)\tau(b^{2}).$

The first really interesting case comes when we compute $\tau(abab).$ We start
with
\[
0=\tau\lbrack(a-\tau(a))(b-\tau(b))(a-\tau(a))(b-\tau(b))]
\]
and expand out the right-hand side as $\tau(abab)$ plus a sum of fifteen
terms, all of which reduce to previously computed quantities. Sparing the
reader the details of this computation, we find that%
\[
\tau(abab)=\tau(a^{2})\tau(b)^{2}+\tau(a)^{2}\tau(b^{2})-\tau(a)^{2}%
\tau(b)^{2}.
\]

Although the notion of free independence will not explicitly be used in the
rest of this article, it is certainly a key concept that is always lurking in
the background.

\subsection{The circular Brownian motion\label{circBrownian.sec}}

If $Z^{N}$ is a Ginibre random matrix (Section \ref{ginibre.sec}), then the
$\ast$-moments of $Z^{N}$ converge to those of a \textquotedblleft circular
element\textquotedblright\ $c$ in a certain tracial von Neumann algebra
$(\mathcal{A},\tau).$ The $\ast$-moments of $c$ can be computed in an
efficient combinatorial way (e.g., Example 11.23 in \cite{NS}). We have, for
example, $\tau(c^{\ast}c)=1$ and $\tau(c^{k})=0$ for all positive integers
$k.$

More generally, we can realize the large-$N$ limit of the entire Ginibre
Brownian motion $C_{t}^{N},$ for all $t>0,$ as a family of elements $c_{t}$ in
a tracial von Neumann algebra $(\mathcal{A},\tau).$ In the limit, the ordinary
independence conditions for the increments of $C_{t}^{N}$ (Section
\ref{ginBrownian.sec}) is replaced by the free independence of the increments
of $c_{t}.$ That is, for all $0=t_{0}<t_{1}<\cdots<t_{k},$ the elements%
\[
c_{t_{1}}-c_{t_{0}},c_{t_{2}}-c_{t_{1}},\ldots,c_{t_{k}}-c_{t_{k-1}}%
\]
are freely independent, in the sense described in the previous subsection. For
any $t>0,$ the $\ast$-distribution of $c_{t}$ is the same as the $\ast
$-distribution of $\sqrt{t}c_{1}.$

\section{Brown measure}

\subsection{The goal}

Recall that if $A$ is an $N\times N$ matrix with eigenvalues $\lambda
_{1},\ldots,\lambda_{N},$ the empirical eigenvalue distribution $\mu_{A}$ of
$A$ is the probability measure on $\mathbb{C}$ assigning mass $1/N$ to each
eigenvalue:%
\[
\mu_{A}=\frac{1}{N}\sum_{j=1}^{N}\delta_{\lambda_{j}}.
\]

\begin{goal}
Given an arbitrary element $x$ in a tracial von Neumann algebra $(\mathcal{A}%
,\tau),$ construct a probability measure $\mu_{x}$ on $\mathbb{C}$ analogous
to the empirical eigenvalue distribution of a matrix.
\end{goal}

If $x\in\mathcal{A}$ is normal, then there is a standard way to construct such
a measure. The spectral theorem allows us to construct a projection-valued
measure $\gamma_{x}$ \cite[Section 10.3]{HallQM} associated to $x$. For each
Borel set $E,$ the projection $\gamma_{x}(E)$ will, again, belong to the von
Neumann algebra $\mathcal{A},$ and we may therefore define%
\begin{equation}
\mu_{x}(E)=\tau\lbrack\gamma_{x}(E)].\label{spectral}%
\end{equation}
We refer to $\mu_{x}$ as the \textbf{distribution} of $x$ (relative to the
trace $\tau$). If $x$ is not normal, we need a different construction---but
one that we hope will agree with the above construction in the normal case.

\subsection{A motivating computation}

If $A$ is an $N\times N$ matrix, define a function $s:\mathbb{C}%
\rightarrow\mathbb{R\cup\{-\infty\}}$ by
\[
s(\lambda)=\log(\left\vert \det(A-\lambda)\right\vert ^{2/N}),
\]
where the logarithm takes the value $-\infty$ when $\det(A-\lambda)=0.$ Note
that $s$ is computed from the characteristic polynomial $\det(A-\lambda)$ of
$A.$ We can compute $s$ in terms of its eigenvalues $\lambda_{1}%
,\ldots,\lambda_{N}$ (taken with their algebraic multiplicity) as%
\begin{equation}
s(\lambda)=\frac{2}{N}\sum_{j=1}^{N}\log\left\vert \lambda-\lambda
_{j}\right\vert . \label{sMatrix}%
\end{equation}
See Figure \ref{splot.fig} for a plot of (the negative of) $s(\lambda).$%

\begin{figure}[ptb]%
\centering
\includegraphics[
height=2.4111in,
width=3.3131in
]%
{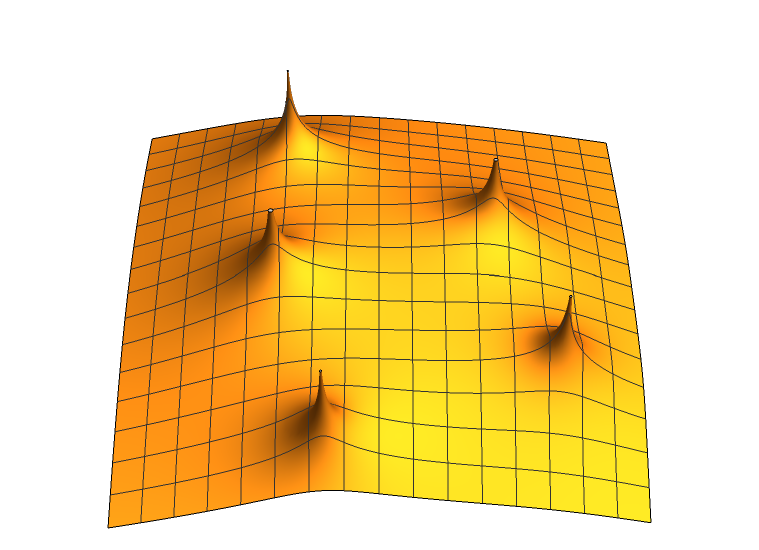}%
\caption{A plot of the function $-s(\lambda)$ for a matrix with five
eigenvalues. The function is harmonic except at the singularities}%
\label{splot.fig}%
\end{figure}

We then recall that the function $\log\left\vert \lambda\right\vert $ is a
multiple of the Green's function for the Laplacian on the plane, meaning that
the function is harmonic away from the origin and that%
\[
\Delta\log\left\vert \lambda\right\vert =2\pi\delta_{0}(\lambda),
\]
where $\delta_{0}$ is a $\delta$-measure at the origin. Thus, if we take the
Laplacian of $s(\lambda),$ with an appropriate normalizing factor, we get the
following nice result.

\begin{proposition}
The Laplacian, in the distribution sense, of the function $s(\lambda)$ in
(\ref{sMatrix}) satisfies%
\[
\frac{1}{4\pi}\Delta s(\lambda)=\frac{1}{N}\sum_{j=1}^{N}\delta_{\lambda_{j}%
}(\lambda),
\]
where $\delta_{\lambda_{j}}$ is a $\delta$-measure at $\lambda_{j}.$ That is
to say, $\frac{1}{4\pi}\Delta s$ is the empirical eigenvalue distribution\ of
$A$ (Definition \ref{eed.def}).
\end{proposition}

Recall that if $B$ is a strictly positive self-adjoint matrix, then we can
take the logarithm of $B,$ which is the self-adjoint matrix obtained by
keeping the eigenvectors of $B$ fixed and taking the logarithm of the eigenvalues.

\begin{proposition}
The function $s$ in (\ref{sMatrix}) can also be computed as%
\begin{equation}
s(\lambda)=\frac{1}{N}\mathrm{trace}[\log((A-\lambda)^{\ast}(A-\lambda))]
\label{sFinite1}%
\end{equation}
or as
\begin{equation}
s(\lambda)=\lim_{\varepsilon\rightarrow0^{+}}\frac{1}{N}\mathrm{trace}%
[\log((A-\lambda)^{\ast}(A-\lambda)+\varepsilon)]. \label{sFinite2}%
\end{equation}
Here the logarithm is the self-adjoint logarithm of a positive self-adjoint matrix.
\end{proposition}

Note that in (\ref{sFinite1}), the logarithm is undefined when $\lambda$ is an
eigenvalue of $A.$ In (\ref{sFinite2}), inserting $\varepsilon>0$ guarantees
that the logarithm is well defined for all $\lambda,$ but a singularity of
$s(\lambda)$ at each eigenvalue still arises in the limit as $\varepsilon$
approaches zero.

\begin{proof}
An elementary result \cite[Theorem 2.12]{HallLie} says that for any matrix
$X,$ we have $\det(e^{X})=e^{\mathrm{trace}(X)}$. If $P$ is a strictly
positive matrix, we may apply this result with $X=\log P$ (so that $e^{X}=P$)
to get%
\[
\det(P)=e^{\mathrm{trace}(X)}%
\]
or%
\[
\mathrm{trace}(\log P)=\log[\det P].
\]
Let us now apply this identity with $P=(A-\lambda)^{\ast}(A-\lambda),$
whenever $\lambda$ is not an eigenvalue of $A,$ to obtain%
\begin{align*}
\frac{1}{N}\mathrm{trace}[\log((A-\lambda)^{\ast}(A-\lambda))]  &  =\frac
{1}{N}\log[\det((A-\lambda)^{\ast}(A-\lambda))]\\
&  =\frac{1}{N}\log[\det(A-\lambda)^{\ast}\det(A-\lambda)]\\
&  =\log(\left\vert \det(A-\lambda)\right\vert ^{2/N}),
\end{align*}
where this last expression is the definition of $s(\lambda).$

Continuity of the matrix logarithm then establishes (\ref{sFinite2}).
\end{proof}

\subsection{Definition and basic properties}

To define the Brown measure of a general element $x$ in a tracial von Neumann
algebra $(\mathcal{A},\tau),$ we use the obvious generalization of
(\ref{sFinite2}). We refer to Brown's original paper \cite{Br} along with
Chapter 11 of \cite{MS} for general references on the material in this section.

\begin{theorem}
\label{BrownMeasure.thm}Let $(\mathcal{A},\tau)$ be a tracial von Neumann
algebra and let $x$ be an arbitrary element of $\mathcal{A}.$ Define%
\begin{equation}
S(\lambda,\varepsilon)=\tau\lbrack\log((x-\lambda)^{\ast}(x-\lambda
)+\varepsilon)]\label{sgeneral1}%
\end{equation}
for all $\lambda\in\mathbb{C}$ and $\varepsilon>0.$ Then
\begin{equation}
s(\lambda):=\lim_{\varepsilon\rightarrow0^{+}}S(\lambda,\varepsilon
)\label{sgeneral2}%
\end{equation}
exists as an almost-everywhere-defined subharmonic function. Furthermore, the
quantity%
\begin{equation}
\frac{1}{4\pi}\Delta s,\label{sgeneral3}%
\end{equation}
where the Laplacian is computed in the distribution sense, is represented by a
probability measure on the plane.\ We call this measure the Brown measure of
$x$ and denote it by $\mu_{x}.$

The Brown measure of $x$ is supported on the spectrum $\sigma(x)$ of $x$ and
has the property that%
\begin{equation}
\int_{\sigma(x)}\lambda^{k}~d\mu_{x}(\lambda)=\tau(x^{k})\label{holoMoment}%
\end{equation}
for all non-negative integers $k.$
\end{theorem}

See the original article \cite{Br} or Chapter 11 of the monograph \cite{MS} of
Mingo and Speicher. We also note that the quantity $s(\lambda)$ is the
logarithm of the \textbf{Fuglede--Kadison determinant} of $x-\lambda$; see
\cite{FK1,FK2}. It is important to emphasize that, in general, the moment
condition (\ref{holoMoment}) \textit{does not uniquely determine the measure
}$\mu_{x}.$ After all, $\sigma(x)$ is an arbitrary nonempty compact subset of
$\mathbb{C},$ which could, for example, be a closed disk. To uniquely
determine the measure, we would need to know the value of $\int_{\sigma
(x)}\lambda^{k}\bar{\lambda}^{l}~d\mu_{x}(\lambda)$ for all non-negative
integers $k$ and $l.$ There is not, however, any simple way to compute the
value of $\int_{\sigma(x)}\lambda^{k}\bar{\lambda}^{l}~d\mu_{x}(\lambda)$ in
terms of the operator $x$. In particular, unless $x$ is normal, this integral
need not be equal to $\tau\lbrack x^{k}(x^{\ast})^{l}].$ Thus, to compute the
Brown measure of a general operator $x\in\mathcal{A},$ we actually have to
work with the rather complicated definition in (\ref{sgeneral1}),
(\ref{sgeneral2}), and (\ref{sgeneral3}).

We note two important special cases.

\begin{itemize}
\item Suppose $\mathcal{A}$ is the space of all $N\times N$ matrices and
$\tau$ is the normalized trace, $\tau\lbrack x]=\frac{1}{N}\mathrm{trace}(x).$
Then the Brown measure of any $x\in\mathcal{A}$ is simply the empirical
eigenvalue distribution of $x,$ which puts mass $1/N$ at each eigenvalue of
$x.$

\item If $x$ is normal, then the Brown measure $\mu_{x}$ of $x$ agrees with
the measure defined in (\ref{spectral}) using the spectral theorem.
\end{itemize}

\subsection{Brown measure in random matrix theory}

Suppose one has a family of $N\times N$ random matrix models $X^{N}$ and one
wishes to determine the large-$N$ limit of the empirical eigenvalue
distribution of $X^{N}.$ (Recall Definition \ref{eed.def}.) One may naturally
use the following three-step process.

\textbf{Step 1}. Construct a large-$N$ limit of $X^{N}$ as an operator $x$ in
a tracial von Neumann algebra $(\mathcal{A},\tau).$

\textbf{Step 2}. Determine the Brown measure $\mu_{x}$ of $x.$

\textbf{Step 3}. Prove that the empirical eigenvalue distribution of $X^{N}$
converges almost surely to $x$ as $N$ tends to infinity.

It is important to emphasize that Step 3 in this process is not automatic.
Indeed, this can be a difficult technical problem. Nevertheless, this article
is concerned with exclusively with Step 2 in the process (in situations where
Step 1 has been carried out). For Step 3, the main tool is the Hermitization
method developed in Girko's pioneering paper \cite{Girko} and further refined
by Bai \cite{Bai}. (Although neither of these authors explicitly uses the
terminology of Brown measure, the idea is lurking there.)

There exist certain pathological examples where the limiting eigenvalue
distribution does not coincide with the Brown measure. In light of a result of
\'{S}niady \cite{Sniady}, we can say that such examples are associated with
spectral instability, that is, matrices where a small change in the matrix
produces a large change in the eigenvalues. \'{S}niady shows that if we add to
$X^{N}$ a small amount of random Gaussian noise, then eigenvalues distribution
of the perturbed matrices will converge to the Brown measure of the limiting
object. (See also the papers \cite{GWZ} and \cite{FPZ}, which obtain similar
results by very different methods.) Thus, if the original random matrices
$X^{N}$ are somehow \textquotedblleft stable,\textquotedblright\ adding this
noise should not change the eigenvalues of $X^{N}$ by much, and the
eigenvalues of the original and perturbed matrices should be almost the same.
In such a case, we should get convergence of the eigenvalues of $X^{N}$ to the
Brown measure of the limiting object.

The canonical example in which instability occurs is the case in which
$X^{N}=\mathrm{nil}_{N},$ the deterministic $N\times N$ matrix having $1$'s
just above the diagonal and 0's elsewhere. Then of course $\mathrm{nil}_{N}$
is nilpotent, so all of its eigenvalues are zero. We note however, that both
$\mathrm{nil}_{N}^{\ast}\mathrm{nil}_{N}$ and $\mathrm{nil}_{N}\mathrm{nil}%
_{N}^{\ast}$ are diagonal matrices whose diagonal entries have $N-1$ values of
1 and only a single value of 0. Thus, when $N$ is large, $\mathrm{nil}_{N}$ is
\textquotedblleft almost unitary,\textquotedblright\ in the sense that
$\mathrm{nil}_{N}^{\ast}\mathrm{nil}_{N}$ and $\mathrm{nil}_{N}\mathrm{nil}%
_{N}^{\ast}$ are close to the identity. Furthermore, for any positive integer
$k,$ we have that $\mathrm{nil}_{N}^{k}$ is again nilpotent, so that
$\mathrm{trace}[\mathrm{nil}_{N}^{k}]=0.$ Using these observations, it is not
hard to show that the limiting object is a \textquotedblleft Haar
unitary,\textquotedblright\ that is, a unitary element $u$ of a tracial von
Neumann algebra satisfying $\tau(u^{k})=0$ for all positive integers $k.$ The
Brown measure of a Haar unitary is the uniform probability measure on the unit
circle, while of course the eigenvalue distribution $X^{N}$ is entirely
concentrated at the origin.

In Figure \ref{perturb1.fig} we see that even under a quite small perturbation
(adding $10^{-6}$ times a Ginibre matrix), the spectrum of the nilpotent
matrix $X^{N}$ changes quite a lot. After the perturbation, the spectrum
clearly resembles a uniform distribution over the unit circle. In Figure
\ref{perturb2.fig}, by contrast, we see that even under a much larger
perturbation (adding $10^{-1}$ times a Ginibre matrix), the spectrum of a GUE
matrix changes only slightly. (Note the vertical scale in Figure
\ref{perturb2.fig}.)%

\begin{figure}[ptb]%
\centering
\includegraphics[
height=2.1179in,
width=4.0283in
]%
{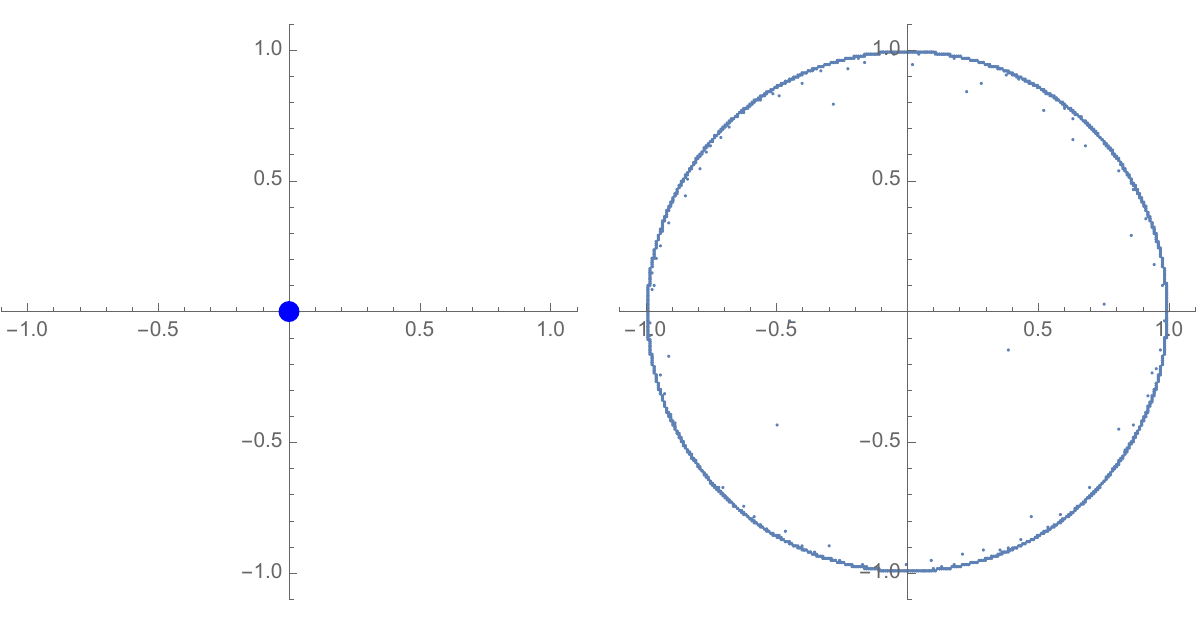}%
\caption{Spectrum of the nilpotent matrix $X$ (left) and $X+\varepsilon
($Ginibre$)$ with $\varepsilon=10^{-5}$ (right), with $N=2,000$}%
\label{perturb1.fig}%
\end{figure}
%

\begin{figure}[ptb]%
\centering
\includegraphics[
height=1.1035in,
width=4.0283in
]%
{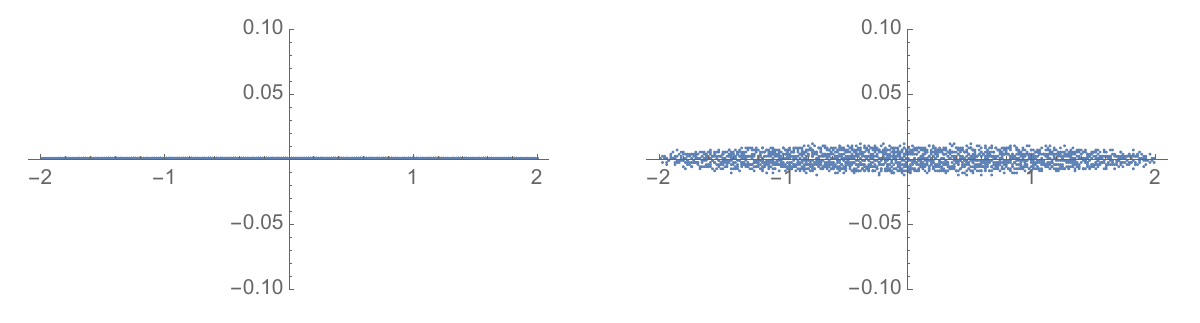}%
\caption{Spectrum of a GUE matrix $X$ (left) and $X+\varepsilon($Ginibre$)$
with $\varepsilon=10^{-1}$ (right), with $N=2,000$}%
\label{perturb2.fig}%
\end{figure}

\subsection{The case of the circular Brownian motion}

We now record the Brown measure of the circular Brownian motion.

\begin{proposition}
\label{BrownCircular.prop}For any $t>0,$ the Brown measure of $c_{t}$ is the
uniform probability measure on the disk of radius $\sqrt{t}$ centered at the origin.
\end{proposition}

Now, as we noted in Section \ref{circBrownian.sec}, the $\ast$-distribution of
the circular Brownian motion at any time~$t>0$ is the same as the $\ast
$-distribution of $\sqrt{t}c_{1}.$ Thus, the proposition will follow if we
know that the Brown measure of a circular element $c$ is the uniform
probability measure on the unit disk. This result, in turn, is well known;
see, for example, Section 11.6.3 of \cite{MS}.

\section{PDE for the circular law}

The PDE methods developed in \cite{DHKBrown} were extended by Ho and Zhong
\cite{HZ} to the case of a free multiplicative Brownian motion with an
arbitrary unitary initial distribution. Ho and Zhong also analyze
\cite[Section 3]{HZ} the case of the free circular Brownian motion with an
arbitrary self-adjoint initial condition. In this article, I explain how the
proof in \cite{HZ} works in the very special case of the free circular
Brownian motion with an initial condition of 0.

The significance of this analysis is not so much that it gives another
computation of the Brown measure of a circular element. Rather, it is a
helpful warm-up case on the path to tackling more complicated problems, such
as those in \cite{DHKBrown} and \cite{HZ}. In this section and the two that
follow, I will show how the PDE\ method applies in the case of the circular
Brownian motion. Then in the last section, I will outline the case of the free
multiplicative Brownian motion, leaving the details to \cite{DHKBrown}.

We let $c_{t}$ be the circular Brownian motion (Section \ref{circBrownian.sec}%
). Then, following the construction of the Brown measure in Theorem
\ref{BrownMeasure.thm}, we define, for each $\lambda\in\mathbb{C},$ a function
$S^{\lambda}$ given by
\begin{equation}
S^{\lambda}(t,\varepsilon)=\tau\lbrack\log((c_{t}-\lambda)^{\ast}%
(c_{t}-\lambda)+\varepsilon)]\label{Sdef}%
\end{equation}
for all $t>0$ and $\varepsilon>0.$ The Brown measure of $c_{t}$ will then be
obtained by letting $\varepsilon$ tend to zero, taking the Laplacian with
respect to $\lambda,$ and dividing by $4\pi.$ The first main
result---following \cite[Section 3]{HZ}---is that, for each $\lambda,$
$S^{\lambda}(t,\varepsilon)$ satisfies a PDE\ in $t$ and $\varepsilon.$

\begin{theorem}
\label{circularPDE.thm}For each $\lambda\in\mathbb{C},$ the function
$S^{\lambda}$ satisfies the first-order, nonlinear differential equation%
\begin{equation}
\frac{\partial S^{\lambda}}{\partial t}=\varepsilon\left(  \frac{\partial
S^{\lambda}}{\partial\varepsilon}\right)  ^{2}\label{additvePDE}%
\end{equation}
subject to the initial condition%
\[
S^{\lambda}(0,\varepsilon)=\log(\left\vert \lambda\right\vert ^{2}%
+\varepsilon).
\]

\end{theorem}

This result (with a more general initial condition) was established in
Proposition 3.1 of \cite{HZ}. We may now see the motivation for making
$\lambda$ a parameter rather than a variable for $S$: since $\lambda$ does not
appear in the PDE (\ref{additvePDE}), we can think of solving the same
equation for each different value of $\lambda,$ with the dependence on
$\lambda$ entering only through the initial conditions.

On the other hand, we see that the regularization parameter $\varepsilon$
plays a crucial role here as one of the variables in our PDE. Of course, we
are ultimately interested in letting $\varepsilon$ tend to zero, but since
derivatives with respect to $\varepsilon$ appear, we cannot merely set
$\varepsilon=0$ in the PDE.

Of course, the reader will point out that, formally, setting $\varepsilon=0$
in (\ref{additvePDE}) gives $\partial S^{\lambda}(t,0)/\partial t=0,$ because
of the leading factor of $\varepsilon$ on the right-hand side. This
conclusion, however, is not actually correct, because $\partial S^{\lambda
}/\partial\varepsilon$ can blow up as $\varepsilon$ approaches zero. Actually,
it will turn out that $S^{\lambda}(t,0)$ \textit{is} independent of $t$ when
$\left\vert \lambda\right\vert >\sqrt{t},$ but not in general.

\subsection{The finite-$N$ equation\label{finiteN.sec}}

In this subsection, we give a heuristic argument for the PDE\ in Theorem
\ref{circularPDE.thm}. Although the argument is not rigorous as written, it
should help explain what is going on. In particular, the computations that
follow should make it clear why the PDE is only valid after taking the
large-$N$ limit.

\subsubsection{The result}

We introduce a finite-$N$ analog of the function $S^{\lambda}$ in Theorem
\ref{circularPDE.thm} and compute its time derivative. Let $C_{t}^{N}$ denote
the Ginibre Brownian motion introduced in Section \ref{ginBrownian.sec}.

\begin{proposition}
\label{finiteN.prop}For each $N,$ let%
\[
S^{\lambda,N}(t,\varepsilon)=\mathbb{E}\{\mathrm{tr}[\log((C_{t}^{N}%
-\lambda)^{\ast}(C_{t}^{N}-\lambda)+\varepsilon)]\}.
\]
Then we have the following results.

\begin{enumerate}
\item \label{timeDeriv.point}The time derivative of $S^{\lambda,N}$ may be
computed as%
\begin{equation}
\frac{\partial S^{\lambda,N}}{\partial t}=\varepsilon\mathbb{E}\{(\mathrm{tr}%
[((C_{t}^{N}-\lambda)^{\ast}(C_{t}^{N}-\lambda)+\varepsilon)^{-1}%
])^{2}\}.\label{dSNdt}%
\end{equation}

\item \label{xDeriv.point}We also have%
\begin{equation}
\frac{\partial}{\partial\varepsilon}\mathrm{tr}[\log((C_{t}^{N}-\lambda
)^{\ast}(C_{t}-\lambda)+\varepsilon)]=\mathrm{tr}[((C_{t}^{N}-\lambda)^{\ast
}(C_{t}^{N}-\lambda)+\varepsilon)^{-1}].\label{dTNdx}%
\end{equation}

\item \label{cov.point}Therefore, if we set%
\[
T^{\lambda,N}=\mathrm{tr}[((C_{t}^{N}-\lambda)^{\ast}(C_{t}^{N}-\lambda
)+\varepsilon)^{-1}],
\]
we may rewrite the formula for $\partial S^{\lambda,N}/\partial t$ as%
\begin{equation}
\frac{\partial S^{\lambda,N}}{\partial t}=\varepsilon\left(  \frac{\partial
S^{\lambda,N}}{\partial\varepsilon}\right)  ^{2}+\varepsilon~\mathrm{Var}%
,\label{derivCov}%
\end{equation}
where $\mathrm{Var}$ is a \textquotedblleft variance term\textquotedblright%
\ given by%
\[
\mathrm{Var}=\mathbb{E}\{(T^{\lambda,N})^{2}\}-(\mathbb{E}\{T^{\lambda
,N}\})^{2}.
\]

\end{enumerate}
\end{proposition}

The key point to observe here is that in the formula (\ref{dSNdt}) for
$\partial S^{\lambda,N}/\partial t$, we have the \textit{expectation value of
the square} of a trace. On the other hand, if we computed $(\partial
S^{\lambda,N}/\partial\varepsilon)^{2}$ by taking the expectation value of
both sides of (\ref{dTNdx}) and squaring, we would have the \textit{square of
the expectation value} of a trace. Thus, there is no PDE for $S^{\lambda,N}%
$---we get an unavoidable covariance term on the right hand side of
(\ref{derivCov}).

On the other hand, the Ginibre Brownian motion $C_{t}^{N}$ exhibits a
\textbf{concentration phenomenon} for large $N.$ Specifically, let us consider
a family $\{Y^{N}\}$ of random variables of the form%
\[
Y^{N}=\mathrm{tr}[\text{word in }C_{t}^{N}\text{ and }(C_{t}^{N})^{\ast}].
\]
(Thus, for example, we might have $Y^{N}=\mathrm{tr}[C_{t}^{N}(C_{t}%
^{N})^{\ast}C_{t}^{N}(C_{t}^{N})^{\ast}].$) Then it is known that (1) the
large-$N$ limit of $\mathbb{E}\{Y^{N}\}$ exists, and (2) the variance of
$Y^{N}$ goes to zero. That is to say, when $N$ is large, $Y^{N}$ will be, with
high probability, close to its expectation value. It then follows that
$\mathbb{E}\{(Y^{N})^{2}\}$ will be close to $(\mathbb{E}\{Y^{N}\})^{2}.$
(This concentration phenomenon was established by Voiculescu in \cite{Voi2}
for the analogous case of the \textquotedblleft GUE Brownian
motion.\textquotedblright\ The case of the Ginibre Brownian motion is similar.)

Now, although the quantity%
\[
((C_{t}^{N}-\lambda)^{\ast}(C_{t}^{N}-\lambda)+\varepsilon)^{-1}%
\]
is not a word in $C_{t}^{N}$ and $(C_{t}^{N})^{\ast},$ it is expressible---at
least for large $\varepsilon$---as a power series in such words. It is
therefore reasonable to expect---this is not a proof!---that the variance of
$X^{N}$ will go to zero as $N$ goes to infinity, and the variance term\ in
(\ref{derivCov}) will vanish in the limit.

\subsubsection{Setting up the computation}

We view $M_{N}(\mathbb{C})$ as a real vector space of dimension $2N^{2}$ and
we use the following real-valued inner product $\left\langle \cdot
,\cdot\right\rangle _{N}$:%
\begin{equation}
\left\langle X,Y\right\rangle _{N}=N\operatorname{Re}(\mathrm{trace}(X^{\ast
}Y)).\label{mnInner}%
\end{equation}
The distribution of $C_{t}^{N}$ is the Gaussian measure of variance $t/2$ with
respect to this inner product%
\[
d\gamma_{t}(C)=d_{t}e^{-\left\langle C,C\right\rangle /t}~dC,
\]
where $d_{t}$ is a normalization constant and $dC$ is the Lebesgue measure on
$M_{N}(\mathbb{C}).$ This measure is a \textbf{heat kernel measure}. If we let
$\mathbb{E}_{t}$ denote the expectation value with respect to $\gamma_{t},$
then we have, for any \textquotedblleft nice\textquotedblright\ function,%
\begin{equation}
\frac{d}{dt}\mathbb{E}_{t}\{f\}=\frac{1}{4}\mathbb{E}_{t}\{\Delta
f\},\label{heatEq}%
\end{equation}
where $\Delta$ is the Laplacian on $M_{N}(\mathbb{C})$ with respect to the
inner product (\ref{mnInner}).

To compute more explicitly, we choose an orthonormal basis for $M_{N}%
(\mathbb{C})$ over $\mathbb{R}$ consisting of $X_{1},\ldots,X_{N^{2}}$ and
$Y_{1},\ldots,Y_{N^{2}},$ where $X_{1},\ldots,X_{N^{2}}$ are skew-Hermitian
and where $Y_{j}=iX_{j}.$ We then introduce the directional derivatives
$\tilde{X}_{j}$ and $\tilde{Y}_{j}$ defined by%
\[
(\tilde{X}_{j}f)(a)=\left.  \frac{d}{ds}f(a+sX_{j})\right\vert _{s=0}%
;\quad(\tilde{Y}_{j}f)(Z)=\left.  \frac{d}{ds}f(a+sY_{j})\right\vert _{s=0}.
\]
Then the Laplacian $\Delta$ is given by%
\[
\Delta=\sum_{j=1}^{N^{2}}\left(  (\tilde{X}_{j})^{2}+(\tilde{Y}_{j}%
)^{2}\right)  .
\]
We also introduce the corresponding complex derivatives, $Z_{j}$ and $\bar
{Z}_{j}$ given by
\begin{align*}
Z_{j}  &  =\frac{1}{2}(\tilde{X}_{j}-i\tilde{Y}_{j});\\
\bar{Z}_{j}  &  =\frac{1}{2}(\tilde{X}_{j}+i\tilde{Y}_{j}),
\end{align*}
which give%
\[
\frac{1}{4}\Delta=\sum_{j=1}^{N^{2}}\bar{Z}_{j}Z_{j}.
\]

We now let $C$ denote a matrix-valued variable ranging over $M_{N}%
(\mathbb{C}).$ We may easily compute the following basic identities:%
\begin{align}
Z_{j}(C) &  =X_{j};\quad Z_{j}(C^{\ast})=0;\nonumber\\
\bar{Z}_{j}(C) &  =0;\quad\bar{Z}_{j}(C^{\ast})=-X_{j}.\label{zjDerivs}%
\end{align}
(Keep in mind that $X_{j}$ is skew-Hermitian.) We will also need the following
elementary but crucial identity%
\begin{equation}
\sum_{j=1}^{N^{2}}X_{j}AX_{j}=-\mathrm{tr}(A),\label{magic}%
\end{equation}
where $\mathrm{tr}(\cdot)$ is the normalized trace, given by
\[
\mathrm{tr}(A)=\frac{1}{N}\mathrm{trace}(A).
\]
See, for example, Proposition 3.1 in \cite{DHKLargeN}. When applied to
function involving a normalized trace, this will produce second trace.

Finally, we need the following formulas for differentiating matrix-valued
functions of a real variable:%
\begin{align}
\frac{d}{ds}A(s)^{-1}  &  =-A(s)^{-1}\frac{dA}{ds}A(s)^{-1}\label{invDeriv}\\
\frac{d}{ds}\mathrm{tr}[\log A(s)]  &  =\mathrm{tr}\left[  A(s)^{-1}\frac
{dA}{ds}\right]  . \label{logDeriv}%
\end{align}
The first of these is standard and can be proved by differentiating the
identity $A(s)A(s)^{-1}=I.$ The second identity is Lemma 1.1 in \cite{Br}; it
is important to emphasize that this second identity does not hold as written
without the trace. One may derive (\ref{logDeriv}) by using an integral
formula for the derivative of the logarithm \textit{without} the trace (see,
for example, Equation (11.10) in \cite{Higham}) and then using the cyclic
invariance of the trace, at which point the integral can be computed explicitly.

\subsubsection{Proof of Proposition \ref{finiteN.prop}}

We continue to let $\mathbb{E}_{t}$ denote the expectation value with respect
to the measure $\gamma_{t},$ which is the distribution at time $t$ of the
Ginibre Brownian motion $C_{t}^{N},$ so that
\[
S^{\lambda,N}(t,\varepsilon)=\mathbb{E}_{t}\{\mathrm{tr}[\log((C-\lambda
)^{\ast}(C-\lambda)+\varepsilon)]\},
\]
where the variable $C$ ranges over $M_{N}(\mathbb{C}).$ We apply the
derivative $Z_{j}$ using (\ref{logDeriv}) and (\ref{zjDerivs}), giving
\[
Z_{j}S^{\lambda,N}(t,\varepsilon)=\mathbb{E}_{t}\{\mathrm{tr}[((C-\lambda
)^{\ast}(C-\lambda)+\varepsilon)^{-1}(C-\lambda)^{\ast}X_{j}]\}.
\]

We then apply the derivative $\bar{Z}_{j}$ using (\ref{invDeriv}) and
(\ref{zjDerivs}), giving%
\begin{align*}
&  \bar{Z}_{j}Z_{j}S^{\lambda,N}(t,\varepsilon)=-\mathbb{E}_{t}\{\mathrm{tr}%
[((C-\lambda)^{\ast}(C-\lambda)+\varepsilon)^{-1}X_{j}^{2}]\}\\
&  +\mathbb{E}_{t}\{\mathrm{tr}[((C-\lambda)^{\ast}(C-\lambda)+\varepsilon
)^{-1}X_{j}(C-\lambda)((C-\lambda)^{\ast}(C-\lambda)+\varepsilon
)^{-1}(C-\lambda)^{\ast}X_{j}]\}.
\end{align*}
We now sum on $j$ and apply the identity\ (\ref{magic}). After applying the
heat equation (\ref{heatEq}) with $\Delta=\sum_{j}\bar{Z}_{j}Z_{j},$ we obtain%
\begin{align}
&  \frac{d}{dt}S^{\lambda,N}(t,\varepsilon)\nonumber\\
&  =\sum_{j}\bar{Z}_{j}Z_{j}S^{\lambda,N}(t,\varepsilon)\nonumber\\
&  =\mathbb{E}_{t}\{\mathrm{tr}[((C-\lambda)^{\ast}(C-\lambda)+\varepsilon
)^{-1}]\}-\mathbb{E}_{t}\{\mathrm{tr}[((C-\lambda)^{\ast}(C-\lambda
)+\varepsilon)^{-1}]\times\nonumber\\
&  \mathrm{tr}[(C-\lambda)^{\ast}(C-\lambda)((C-\lambda)^{\ast}(C-\lambda
)+\varepsilon)^{-1}]\}.\label{dSdt}%
\end{align}

But then
\begin{align*}
&  (C-\lambda)^{\ast}(C-\lambda)((C-\lambda)^{\ast}(C-\lambda)+\varepsilon
)^{-1}\\
&  =((C-\lambda)^{\ast}(C-\lambda)+\varepsilon-\varepsilon)((C-\lambda)^{\ast
}(C-\lambda)+\varepsilon)^{-1}\\
&  =1-\varepsilon((C-\lambda)^{\ast}(C-\lambda)+\varepsilon)^{-1}.
\end{align*}
Thus, there is a cancellation between the two terms on the right-hand side of
(\ref{dSdt}), giving%
\[
\frac{\partial S^{\lambda,N}}{\partial t}=\varepsilon\mathbb{E}_{t}%
\{(\mathrm{tr}[((C-\lambda)^{\ast}(C-\lambda)+\varepsilon)^{-1}])^{2}\},
\]
as claimed in Point \ref{timeDeriv.point} of the proposition.

Meanwhile, we may use again the identity (\ref{logDeriv}) to compute%
\[
\frac{\partial}{\partial\varepsilon}\mathrm{tr}[\log((C_{t}^{N}-\lambda
)^{\ast}(C_{t}-\lambda)+\varepsilon)]
\]
to verify Point \ref{xDeriv.point} of the proposition. Point \ref{cov.point}
then follows by simple algebra.

\subsection{A derivation using free stochastic calculus}

\subsubsection{Ordinary stochastic calculus}

In this section, I\ will describe briefly how the PDE in Theorem
\ref{circularPDE.thm} can be derived rigorously, using the tools of free
stochastic calculus. The proof will follow Section 3 of \cite{HZ}.

We begin by recalling a little bit of ordinary stochastic calculus, for the
ordinary, real-valued Brownian motion. To avoid notational conflicts, we will
let $x_{t}$ denote Brownian motion in the real line. This is a random
continuous path satisfying the properties proposed by Einstein in 1905, namely
that for any $0=t_{0}<t_{1}<\cdots<t_{k},$ the increments%
\[
x_{t_{1}}-x_{t_{0}},~x_{t_{2}}-x_{t_{1}},\ldots,~x_{t_{k}}-x_{t_{k-1}}%
\]
should be independent normal random variables with mean zero and variance
$t_{j}-t_{j-1}.$ At a rigorous level, Brownian motion is described by the
Wiener measure on the space of continuous paths.

It is a famous result that, with probability one, the path $x_{t}$ is nowhere
differentiable. This property has not, however, deterred people from
developing a theory of \textquotedblleft stochastic calculus\textquotedblright%
\ in which one can take the \textquotedblleft differential\textquotedblright%
\ of $x_{t},$ denoted $dx_{t}.$ (Since $x_{t}$ is not differentiable, we
should not attempt to rewrite this differential as $\frac{dx_{t}}{dt}dt.$)
There is then a theory of \textquotedblleft stochastic
integrals,\textquotedblright\ in which one can compute, for example, integrals
of the form%
\[
\int_{a}^{b}f(x_{t})~dx_{t},
\]
where $f$ is some smooth function.

A key difference between ordinary and stochastic integration is that
$(dx_{t})^{2}$ is not negligible compared to $dt.$ To understand this
assertion, recall that the increments of Brownian motion have variance
$t_{j}-t_{j-1}$---and therefore standard deviation $\sqrt{t_{j}-t_{j-1}}.$
This means that in a short time interval $\Delta t,$ the Brownian motion
travels distance roughly $\Delta t.$ Thus, if $\Delta x_{t}=x_{t+\Delta
t}-x_{t},$ we may say that $(\Delta x_{t})^{2}\approx\Delta t.$ Thus, if $f$
is a smooth function, we may use a Taylor expansion to claim that%
\begin{align*}
f(x_{t+\Delta t})  &  \approx f(x_{t})+f^{\prime}(x_{t})\Delta x_{t}+\frac
{1}{2}f^{\prime\prime}(x_{t})(\Delta x_{t})^{2}\\
&  \approx f(x_{t})+f^{\prime}(x_{t})\Delta x_{t}+\frac{1}{2}f^{\prime\prime
}(x_{t})\Delta t.
\end{align*}
We may express the preceding discussion in the heuristically by saying%
\[
(dx_{t})^{2}=dt.
\]

Rigorously, this line of reasoning lies behind the famous It\^{o} formula,
which says that
\[
df(x_{t})=f^{\prime}(x_{t})~dx_{t}+\frac{1}{2}f^{\prime\prime}(x_{t})~dt.
\]
The formula means, more precisely, that (after integration)%
\[
f(x_{b})-f(x_{a})=\int_{a}^{b}f^{\prime}(x_{t})~dx_{t}+\frac{1}{2}\int_{a}%
^{b}f^{\prime\prime}(x_{t})~dt,
\]
where the first integral on the right-hand side is a stochastic integral and
the second is an ordinary Riemann integral.

If we take, for example, $f(x)=x^{2}/2,$ then we find that%
\[
\frac{1}{2}(x_{b}^{2}-x_{a}^{2})=\int_{a}^{b}x_{t}~dx_{t}+\frac{1}{2}(b-a)
\]
so that%
\[
\int_{a}^{b}x_{t}~dx_{t}=\frac{1}{2}(x_{b}^{2}-x_{a}^{2})-\frac{1}{2}(b-a).
\]
This formula differs from what we would get if $x_{t}$ were smooth by the
$b-a$ term on the right-hand side.

\subsubsection{Free stochastic calculus}

We now turn to the case of the circular Brownian motion $c_{t}.$ Since $c_{t}$
is a limit of ordinary Brownian motion in the space of $N\times N$ matrices,
we expect that $(dc_{t})^{2}$ will be non-negligible compared to $dt.$ The
rules are as follows; see \cite[Lemma 2.5, Lemma 4.3]{KempLargeN}. Suppose
$g_{t}$ and $h_{t}$ are processes \textquotedblleft adapted to $c_{t}%
$,\textquotedblright\ meaning that $g_{t}$ and $h_{t}$ belong to the von
Neumann algebra generated by the operators $c_{s}$ with $0<s<t.$ Then we have
\begin{align}
dc_{t}\,g_{t}\,dc_{t}^{\ast}  &  =dc_{t}^{\ast}\,g_{t}\,dc_{t}=\tau
(g_{t})\,dt\label{Ito1}\\
dc_{t}\,g_{t}\,dc_{t}  &  =dc_{t}^{\ast}\,g_{t}\,dc_{t}^{\ast}=0\label{Ito2}\\
\tau(g_{t}\,dc_{t}\,h_{t})  &  =\tau(g_{t}\,dc_{t}^{\ast}\,h_{t})=0.
\label{Ito4}%
\end{align}
In addition, we have the following It\^{o} product rule: if $a_{t}^{1}%
,\ldots,a_{t}^{n}$ are processes adapted to $c_{t}$, then
\begin{align}
d(a_{t}^{1}\cdots a_{t}^{n})  &  =\sum_{j=1}^{n}(a_{t}^{1}\cdots a_{t}%
^{j-1})\,da_{t}^{j}\,(a_{t}^{j+1}\cdots a_{t}^{n})\label{productRule1}\\
&  +\sum_{1\leq j<k\leq n}(a_{t}^{1}\cdots a_{t}^{j-1})\,da_{t}^{j}%
\,(a_{t}^{j+1}\cdots a_{t}^{k-1})\,da_{t}^{k}\,(a_{t}^{k+1}\cdots a_{t}^{n}).
\label{productRule2}%
\end{align}
Finally, the differential \textquotedblleft$d$\textquotedblright\ can be moved
inside the trace $\tau.$

Suppose, for example, we wish to compute $d\tau\lbrack c_{t}^{\ast}c_{t}].$ We
start by applying the product rule in (\ref{productRule1}) and
(\ref{productRule2}). But by (\ref{Ito4}), there will be no contribution from
the first line (\ref{productRule1}) in the product rule. We then use the
second line (\ref{productRule2}) of the product rule together with
(\ref{Ito1}) to obtain%
\[
d\tau\lbrack c_{t}^{\ast}c_{t}]=\tau\lbrack dc_{t}^{\ast}dc_{t}]=\tau
(1)~dt=dt.
\]
Thus,
\[
\frac{d}{dt}\tau\lbrack c_{t}^{\ast}c_{t}]=1.
\]
Since, also, $c_{0}=0,$ we find that $\tau\lbrack c_{t}^{\ast}c_{t}]=t.$

\subsubsection{The proof}

In the proof that follows, the It\^{o} formula (\ref{Ito1}) plays the same
role as the identity (\ref{magic}) plays in the heuristic argument in Section
\ref{finiteN.sec}. We begin with a lemma whose proof is an exercise in using
the rules of free stochastic calculus.

\begin{lemma}
\label{nthPower.lem}For each $\lambda\in\mathbb{C},$ let us use the the
notation%
\[
c_{t,\lambda}:=c_{t}-\lambda.
\]
Then for each positive integer $n,$ we have%
\[
\frac{d}{dt}\tau\lbrack(c_{t,\lambda}^{\ast}c_{t,\lambda})^{n}]=n\sum
_{l=0}^{n-1}\tau\lbrack(c_{t,\lambda}^{\ast}c_{t,\lambda})^{j}]\tau
\lbrack(c_{t,\lambda}c_{t,\lambda}^{\ast})^{n-j-1}]
\]

\end{lemma}

\begin{proof}
We first note that $dc_{t,\lambda}=dc_{t}$ and $dc_{t,\lambda}^{\ast}%
=dc_{t}^{\ast},$ since $\lambda$ is a constant. We then compute $d\tau
\lbrack(c_{t,\lambda}^{\ast}c_{t,\lambda})^{n}]$ by moving the $d$ inside the
trace and then applying the product rule in (\ref{productRule1}) and
(\ref{productRule2}). By (\ref{Ito4}), the terms arising from
(\ref{productRule1}) will not contribute. Furthermore, by (\ref{Ito2}), the
only terms from (\ref{productRule2}) that contribute are those where one $d$
goes on a factor of $c_{t,\lambda}$ and one goes on a factor of $c_{t,\lambda
}^{\ast}.$

By choosing all possible factors of $c_{t,\lambda}$ and all possible factors
of $c_{t,\lambda}^{\ast},$ we get $n^{2}$ terms. In each term, after putting
the $d$ inside the trace, we can cyclically permute the factors until, say,
the $dc_{t,\lambda}$ factor is at the end. There are then only $n$
\textit{distinct} terms that occur, each of which occurs $n$ times. By
(\ref{Ito1}), each distinct term is computed as%
\begin{align*}
&  \tau\lbrack(c_{t,\lambda}^{\ast}c_{t,\lambda})^{j}~dc_{t}^{\ast
}c_{t,\lambda}(c_{t,\lambda}^{\ast}c_{t,\lambda})^{n-j-2}c_{t,\lambda}^{\ast
}~dc_{t}]\\
&  =\tau\lbrack c_{t,\lambda}(c_{t,\lambda}^{\ast}c_{t,\lambda})^{n-j-2}%
c_{t,\lambda}^{\ast}]\tau\lbrack(c_{t,\lambda}^{\ast}c_{t,\lambda})^{j}]~dt\\
&  =\tau\lbrack(c_{t,\lambda}^{\ast}c_{t,\lambda})^{j}]\tau\lbrack c_{t}%
c_{t}^{\ast}(c_{t,\lambda}c_{t,\lambda}^{\ast})^{n-j-1}]~dt.
\end{align*}
Since each distinct term occurs $n$ times, we obtain%
\[
d\tau\lbrack(c_{t,\lambda}^{\ast}c_{t,\lambda})^{n}]=n\sum_{j=0}^{n-1}%
\tau\lbrack(c_{t,\lambda}^{\ast}c_{t,\lambda})^{j}]\tau\lbrack(c_{t,\lambda
}c_{t,\lambda}^{\ast})^{n-j-1}]~dt,
\]
which is equivalent to the claimed formula.
\end{proof}

We are now ready to give a rigorous argument for the PDE.

\begin{proof}
[Proof of Theorem \ref{circularPDE.thm}]We continue to use the notation
$c_{t,\lambda}:=c_{t}-\lambda.$ We first compute, using the operator version
of (\ref{logDeriv}), that%
\begin{align}
\frac{\partial S}{\partial\varepsilon} &  =\frac{\partial}{\partial
\varepsilon}\tau\lbrack\log(c_{t,\lambda}^{\ast}c_{t,\lambda}+\varepsilon
)]\nonumber\\
&  =\tau\lbrack(c_{t,\lambda}^{\ast}c_{t,\lambda}+\varepsilon)^{-1}%
].\label{dSdx}%
\end{align}

We note that the definition of $S$ in (\ref{Sdef}) actually makes sense for
all $\varepsilon\in\mathbb{C}$ with $\operatorname{Re}(\varepsilon)>0,$ using
the standard branch of the logarithm function. We note that for $\left\vert
\varepsilon\right\vert >\left\vert z\right\vert ,$ we have%
\begin{align}
\frac{1}{z+\varepsilon} &  =\frac{1}{\varepsilon\left(  1-\left(  -\frac
{z}{\varepsilon}\right)  \right)  }\nonumber\\
&  =\frac{1}{\varepsilon}\left[  1-\frac{z}{\varepsilon}+\frac{z^{2}%
}{\varepsilon^{2}}-\frac{z^{3}}{\varepsilon^{3}}+\cdots\right]  .\label{1zx}%
\end{align}
Integrating with respect to $z$ gives%
\[
\log(z+\varepsilon)=\log\varepsilon+\sum_{n=1}^{\infty}\frac{(-1)^{n-1}}%
{n}\left(  \frac{z}{\varepsilon}\right)  ^{n}.
\]
Thus, for $\left\vert \varepsilon\right\vert >\left\Vert c_{t}^{\ast}%
c_{t}\right\Vert ,$ we have%
\begin{equation}
\tau\lbrack\log(c_{t,\lambda}^{\ast}c_{t,\lambda}+\varepsilon)]=\log
\varepsilon+\,\sum_{n=1}^{\infty}\frac{(-1)^{n-1}}{n\varepsilon^{n}}%
\tau\lbrack(c_{t,\lambda}^{\ast}c_{t,\lambda})^{n}].\label{tauLog}%
\end{equation}

Assume for the moment that it is permissible to differentiate (\ref{tauLog})
term by term with respect to $t.$ Then by Lemma \ref{nthPower.lem}, we have%
\begin{equation}
\frac{\partial S}{\partial t}=\sum_{n=1}^{\infty}\frac{(-1)^{n-1}}%
{\varepsilon^{n}}\sum_{j=0}^{n-1}\tau\lbrack(c_{t,\lambda}^{\ast}c_{t,\lambda
})^{j}]\tau\lbrack(c_{t,\lambda}c_{t,\lambda}^{\ast})^{n-j-1}].\label{dsdt1}%
\end{equation}
Now, by \cite[Proposition 3.2.3]{BS1}, the map $t\mapsto c_{t}$ is continuous
in the operator norm topology; in particular, $\left\Vert c_{t}\right\Vert $
is a locally bounded function of $t.$ From this observation, it is easy to see
that the right-hand side of (\ref{dsdt1}) converges locally uniformly in $t.$
Thus, a standard result about interchange of limit and derivative (e.g.,
Theorem 7.17 in \cite{blueRudin}) shows that the term-by-term differentiation
is valid.

Now, in (\ref{dsdt1}), we let $k=j$ and $l=n-j-1,$ so that $n=k+l+1.$ Then $k$
and $l$ go from 0 to $\infty,$ and we get%
\[
\frac{\partial S}{\partial t}=\varepsilon\left(  \frac{1}{\varepsilon}%
\sum_{k=0}^{\infty}\frac{(-1)^{k}}{\varepsilon^{k}}\tau\lbrack(c_{t,\lambda
}^{\ast}c_{t,\lambda})^{k}]\right)  \left(  \frac{1}{\varepsilon}\sum
_{l=0}^{\infty}\frac{(-1)^{l}}{\varepsilon^{l}}\tau\lbrack(c_{t,\lambda
}c_{t,\lambda}^{\ast})^{l}]\right)  .
\]
(We may check that the power of $\varepsilon$ in the denominator is $k+l+1=n$
and that the power of $-1$ is $k+l=n-1.$) Thus, moving the sums inside the
traces and using (\ref{1zx}), we obtain that%
\begin{equation}
\frac{\partial S}{\partial t}=\varepsilon(\tau\lbrack(c_{t,\lambda}^{\ast
}c_{t,\lambda}+\varepsilon)^{-1}])^{2},\label{dsdt0}%
\end{equation}
which reduces to the claimed PDE\ for $S,$ by (\ref{dSdx}).

We have now established the claimed formula for $\partial S/\partial t$ for
$\varepsilon$ in the right half-plane, provided $\left\vert \varepsilon
\right\vert $ is sufficiently large, depending on $t$ and $\lambda.$ Since,
also, $S(0,\lambda,\varepsilon)=\log(\left\vert \lambda-1\right\vert
^{2}+\varepsilon),$ we have, for sufficiently large $\left\vert \varepsilon
\right\vert ,$%
\begin{equation}
S(t,\lambda,\varepsilon)=\log(\left\vert \lambda-1\right\vert ^{2}%
+\varepsilon)+\int_{0}^{t}\varepsilon\tau\lbrack(c_{s,\lambda}^{\ast
}c_{s,\lambda}+\varepsilon)^{-1}]\tau\lbrack(c_{s,\lambda}c_{s,\lambda}^{\ast
}+\varepsilon)^{-1}]~ds.\label{Sintegrated}%
\end{equation}
We now claim that both sides of (\ref{Sintegrated}) are well-defined,
holomorphic functions of $\varepsilon,$ for $\varepsilon$ in the right
half-plane. This claim is easily established from the standard power-series
representation of the inverse:%
\begin{align*}
(A+\varepsilon+h)^{-1} &  =(A+\varepsilon)^{-1}(1+h(A+\varepsilon)^{-1}%
)^{-1}\\
&  =(A+\varepsilon)^{-1}\sum_{n=0}^{\infty}(-1)^{n}h^{n}(A+\varepsilon)^{-n},
\end{align*}
and a similar power-series representation of the logarithm. Thus,
(\ref{Sintegrated}) actually holds for all $\varepsilon$ in the right
half-plane. Differentiating with respect to $t$ then establishes the desired
formula (\ref{dsdt0}) for $dS/dt$ for all $\varepsilon$ in the right half-plane.
\end{proof}

\section{Solving the equation\label{solving.sec}}

\subsection{The Hamilton--Jacobi method\label{hjMethod.sec}}

The PDE (\ref{additvePDE}) in Theorem \ref{circularPDE.thm} is a first-order,
nonlinear equation of Hamilton--Jacobi type. \textquotedblleft
Hamilton--Jacobi type\textquotedblright\ means that the right-hand side of the
equation involves only $\varepsilon$ and $\partial S/\partial\varepsilon,$ and
not $S$ itself. The reader may consult Section 3.3 of the book \cite{Evans} of
Evens for general information about equations of this type. In this
subsection, we describe the general version of this method. In the remainder
of this section, we will then apply the general method to the PDE
(\ref{additvePDE}).

The Hamilton--Jacobi method for analyzing solutions to equations of this type
is a generalization of the method of characteristics. In the method of
characteristics, one finds certain special curves along which the solution is
constant. For a general equation of Hamilton--Jacobi type, the method of
characteristics in not applicable. Nevertheless, we may hope to find certain
special curves along which the solution varies in a simple way, allowing us to
compute the solution along these curves in a more-or-less explicit way.

We now explain the representation formula for solutions of equations of
Hamilton--Jacobi type. A self-contained proof of the following result is given
as the proof of Proposition 6.3 in \cite{DHKBrown}.

\begin{proposition}
\label{HJgeneral.prop}Fix a function $H(\mathbf{x},\mathbf{p})$ defined for
$\mathbf{x}$ in an open set $U\subset\mathbb{R}^{n}$ and $\mathbf{p}$ in
$\mathbb{R}^{n}.$ Consider a smooth function $S(t,\mathbf{x})$ on
$[0,\infty)\times U$ satisfying%
\begin{equation}
\frac{\partial S}{\partial t}=-H(\mathbf{x},\nabla_{\mathbf{x}}S)
\label{HJeqn}%
\end{equation}
for $\mathbf{x}\in U$ and $t>0.$ Now suppose $(\mathbf{x}(t),\mathbf{p}(t))$
is curve in $U\times\mathbb{R}^{n}$ satisfying Hamilton's equations:%
\[
\frac{dx_{j}}{dt}=\frac{\partial H}{\partial p_{j}}(\mathbf{x}(t),\mathbf{p}%
(t));\quad\frac{dp_{j}}{dt}=-\frac{\partial H}{\partial x_{j}}(\mathbf{x}%
(t),\mathbf{p}(t))
\]
with initial conditions%
\begin{equation}
\mathbf{x}(0)=\mathbf{x}_{0};\quad\mathbf{p}(0)=\mathbf{p}_{0}:=(\nabla
_{\mathbf{x}}S)(0,\mathbf{x}_{0}). \label{HJinit}%
\end{equation}
Then we have%
\begin{equation}
S(t,\mathbf{x}(t))=S(0,\mathbf{x}_{0})-H(\mathbf{x}_{0},\mathbf{p}_{0}%
)~t+\int_{0}^{t}\mathbf{p}(s)\cdot\frac{d\mathbf{x}}{ds}~ds
\label{HJformulaGen}%
\end{equation}
and%
\begin{equation}
(\nabla_{\mathbf{x}}S)(t,\mathbf{x}(t))=\mathbf{p}(t). \label{derivFormulaGen}%
\end{equation}

\end{proposition}

We emphasize that we are not using the Hamilton--Jacobi formula to
\textit{construct} a solution to the equation (\ref{HJeqn}); rather, we are
using the method to \textit{analyze} a solution that is assumed ahead of time
to exist. Suppose we want to use the method to compute (as explicitly as
possible), the value of $S(t,\mathbf{x})$ for some fixed $\mathbf{x}.$ We then
need to try to choose the initial position $\mathbf{x}_{0}$ in (\ref{HJinit}%
)---which determines the initial momentum $\mathbf{p}_{0}=(\nabla_{\mathbf{x}%
}S)(0,\mathbf{x}_{0})$---so that $\mathbf{x}(t)=\mathbf{x}.$ We then use
(\ref{HJformulaGen})\ to get an in-principle formula for $S(t,\mathbf{x}%
(t))=S(t,\mathbf{x}).$

\subsection{Solving the equations}

The equation for $S^{\lambda}$ in Theorem \ref{circularPDE.thm} is of
Hamilton--Jacobi form with $n=1$, with Hamiltonian given by%
\begin{equation}
H(\varepsilon,p)=-\varepsilon p^{2}.\label{Hadditive}%
\end{equation}
Since $S^{\lambda}(t,\varepsilon)$ is only defined for $\varepsilon>0,$ we
take open set $U$ in Proposition \ref{HJgeneral.prop} to be $(0,\infty).$ That
is to say, the Hamilton--Jacobi formula (\ref{HJformulaGen}) is only valid if
the curve $\varepsilon(s)$ remains positive for $0\leq s\leq t.$

Hamilton's equations for this Hamiltonian then take the explicit form%
\begin{align}
\frac{d\varepsilon}{dt} &  =\frac{\partial H}{\partial p}=-2\varepsilon
p\label{Ham1}\\
\frac{dp}{dt} &  =-\frac{\partial H}{\partial\varepsilon}=p^{2}.\label{Ham2}%
\end{align}
Following the general method, we take an arbitrary initial position
$\varepsilon_{0}$, with the initial momentum $p_{0}$ given by
\begin{align}
p_{0} &  =\left.  \frac{\partial}{\partial\varepsilon}\log(\left\vert
\lambda\right\vert ^{2}+\varepsilon)\right\vert _{\varepsilon=\varepsilon_{0}%
}\nonumber\\
&  =\frac{1}{\left\vert \lambda\right\vert ^{2}+\varepsilon_{0}}%
.\label{p0circular}%
\end{align}

\begin{theorem}
\label{odeSoln.thm}For any $\varepsilon_{0}>0,$ the solution $(\varepsilon
(t),p(t))$ to (\ref{Ham1}) and (\ref{Ham2}) with initial momentum
$p_{0}=1/(\left\vert \lambda\right\vert ^{2}+\varepsilon_{0})$ exists for
$0\leq t<\left\vert \lambda\right\vert ^{2}+\varepsilon_{0}.$ On this time
interval, we have%
\begin{equation}
\varepsilon(t)=\varepsilon_{0}\left(  1-\frac{t}{\left\vert \lambda\right\vert
^{2}+\varepsilon_{0}}\right)  ^{2}.\label{xOfT}%
\end{equation}
The general Hamilton--Jacobi formula (\ref{HJformulaGen}) then takes the form%
\begin{align}
&  S^{\lambda}\left(  t,\varepsilon_{0}\left(  1-\frac{t}{\left\vert
\lambda\right\vert ^{2}+\varepsilon_{0}}\right)  ^{2}\right)  \nonumber\\
&  =\log(\left\vert \lambda\right\vert ^{2}+\varepsilon_{0})-\frac
{\varepsilon_{0}t}{(\left\vert \lambda\right\vert ^{2}+\varepsilon_{0})^{2}%
},\quad0\leq t<\left\vert \lambda\right\vert ^{2}+\varepsilon_{0}%
.\label{HJcircular}%
\end{align}

\end{theorem}

\begin{proof}
Since the equation (\ref{Ham2}) for $dp/dt$ does not involve $\varepsilon(t),$
we may easily solve it for $p(t)$ as
\[
p(t)=\frac{p_{0}}{1-p_{0}t}.
\]
We may then plug the formula for $p(t)$ into the equation (\ref{Ham1}) for
$d\varepsilon/dt,$ giving%
\[
\frac{d\varepsilon}{dt}=-2\varepsilon\frac{p_{0}}{1-p_{0}t}%
\]
so that%
\[
\frac{1}{\varepsilon}d\varepsilon=-2\frac{p_{0}}{1-p_{0}t}~dt.
\]
Thus,%
\[
\log\varepsilon=2\log(p_{0}t-1)+c_{1}%
\]
so that%
\[
\varepsilon=c_{2}(1-p_{0}t)^{2}.
\]
Plugging in $t=0$ gives $c_{2}=\varepsilon_{0}.$ Recalling the expression
(\ref{p0circular}) for $p_{0}$ gives the claimed formula for $\varepsilon(t).$

Assuming $\varepsilon_{0}>0,$ the solution to the system (\ref{Ham1}%
)--(\ref{Ham2}) continues to exist with $\varepsilon(t)>0$ until $p(t)$ blows
up, which occurs at time $t=1/p_{0}=\left\vert \lambda\right\vert
^{2}+\varepsilon_{0}.$

Finally, we work out the general Hamilton--Jacobi formula (\ref{HJformulaGen})
in the case at hand. We note from (\ref{Hadditive}) and (\ref{Ham1}) that
$p(s)\frac{d\varepsilon}{ds}=-2\varepsilon(s)p(s)^{2}=2H(s).$ Since the
Hamiltonian is always a conserved quantity in Hamilton's equations, we find
that
\[
p(s)\frac{d\varepsilon}{ds}=2H(0)=-2\varepsilon_{0}p_{0}^{2}.
\]
Thus, (\ref{HJformulaGen}) reduces to
\begin{align*}
S^{\lambda}(t,\varepsilon(t)) &  =S(0,\varepsilon_{0})+H(0)t\\
&  =\log(\left\vert \lambda\right\vert ^{2}+\varepsilon_{0})-\varepsilon
_{0}p_{0}^{2}t.
\end{align*}
Using the formula (\ref{p0circular}) for $p_{0}$ gives the claimed formula
(\ref{HJcircular}).
\end{proof}

\section{Letting $\varepsilon$ tend to zero\label{xToZero.sec}}

Recall that the Brown measure is obtained by first evaluating
\[
s_{t}(\lambda):=\lim_{\varepsilon\rightarrow0^{+}}S^{\lambda}(t,\varepsilon)
\]
and then taking $1/(4\pi)$ times the Laplacian (in the distribution sense) of
$s_{t}(\lambda).$ We record the result here and will derive it in the
remainder of this section.

\begin{theorem}
We have%
\begin{equation}
s_{t}(\lambda)=\left\{
\begin{array}
[c]{cc}%
\log(\left\vert \lambda\right\vert ^{2}) & \left\vert \lambda\right\vert
\geq\sqrt{t}\\
\log t-1+\frac{\left\vert \lambda\right\vert ^{2}}{t} & \left\vert
\lambda\right\vert <\sqrt{t}%
\end{array}
\right.  . \label{stFormula}%
\end{equation}
The Brown measure is then absolutely continuous with respect to the Lebesgue
measure, with density $W_{t}(\lambda)$ given by%
\begin{equation}
W_{t}(\lambda)=\left\{
\begin{array}
[c]{cc}%
0 & \left\vert \lambda\right\vert \geq\sqrt{t}\\
\frac{1}{\pi t} & \left\vert \lambda\right\vert <\sqrt{t}%
\end{array}
\right.  . \label{WtFormula}%
\end{equation}

\end{theorem}

That is to say, the Brown measure is the uniform probability measure on the
disk of radius $\sqrt{t}$ centered at the origin. The functions $s_{t}%
(\lambda)$ and $W_{t}(\lambda)$ are plotted for $t=1$ in Figure
\ref{splot3d.fig}. On the left-hand side of the figure, the dashed line
indicates the boundary of the unit disk.%

\begin{figure}[ptb]%
\centering
\includegraphics[
height=2.5356in,
width=4.8274in
]%
{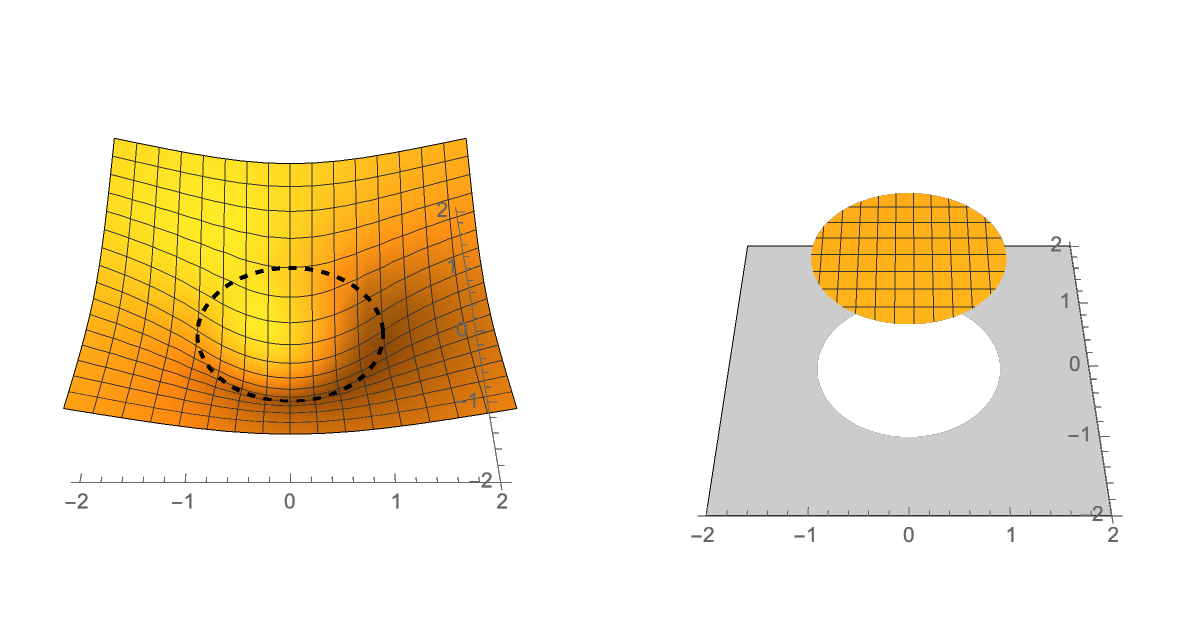}%
\caption{Plot of $s_{t}(\lambda):=S^{\lambda}(t,0^{+})$ (left) and $\frac
{1}{4\pi}\Delta s_{t}(\lambda)$ (right) for $t=1$}%
\label{splot3d.fig}%
\end{figure}

\subsection{Letting $\varepsilon$ tend to zero: outside the disk}

Our goal is to compute $s_{t}(\lambda):=\lim_{\varepsilon\rightarrow0^{+}%
}S^{\lambda}(t,\varepsilon)$. Thus, in the Hamilton--Jacobi formalism, we want
to try to choose $\varepsilon_{0}$ so that the quantity
\begin{equation}
\varepsilon(t)=\varepsilon_{0}\left(  1-\frac{t}{\left\vert \lambda\right\vert
^{2}+\varepsilon_{0}}\right)  ^{2}\label{xOfT2}%
\end{equation}
will be very close to zero. Since there is a factor of $\varepsilon_{0}$ on
the right-hand side of the above formula, an obvious strategy is to take
$\varepsilon_{0}$ itself very close to zero. There is, however, a potential
difficulty with this strategy: If $\varepsilon_{0}$ is small, the lifetime of
the solution may be smaller than the time $t$ we are interested in. To see
when the strategy works, we take the formula for the lifetime of the
solution---namely $\left\vert \lambda\right\vert ^{2}+\varepsilon_{0}$---and
take the limit as $\varepsilon_{0}$ tends to zero.

\begin{definition}
For each $\lambda\in\mathbb{C},$ we define $T(\lambda)$ to be the lifetime of
solutions to the system (\ref{Ham1})--(\ref{Ham2}), in the limit as
$\varepsilon_{0}$ approaches zero. Thus, explicitly,%
\begin{align*}
T(\lambda) &  =\lim_{\varepsilon_{0}\rightarrow0^{+}}(\left\vert
\lambda\right\vert ^{2}+\varepsilon_{0})\\
&  =\left\vert \lambda\right\vert ^{2}.
\end{align*}

\end{definition}

Thus, if the time $t$ we are interested in is larger than $T(\lambda
)=\left\vert \lambda\right\vert ^{2},$ our simple strategy of taking
$\varepsilon_{0}\approx0$ will not work. After all, if $t>T(\lambda)$ and
$\varepsilon_{0}\approx0,$ then the lifetime of the path is less than $t$ and
the Hamilton--Jacobi formula (\ref{HJcircular}) is not applicable. On the
other hand, if the time $t$ we are interested in is at most $T(\lambda
)=\left\vert \lambda\right\vert ^{2},$ the simple strategy does work. Figure
\ref{outside.fig} illustrates the situation.

\begin{conclusion}
The simple strategy of letting $\varepsilon_{0}$ approach zero works precisely
when $t\leq T(\lambda)=\left\vert \lambda\right\vert ^{2}.$ Equivalently, the
simple strategy works when $\left\vert \lambda\right\vert \geq t,$ that is,
when $\lambda$ is outside the open disk of radius $\sqrt{t}$ centered at the origin.
\end{conclusion}

%

\begin{figure}[ptb]%
\centering
\includegraphics[
height=2.0289in,
width=3.7775in
]%
{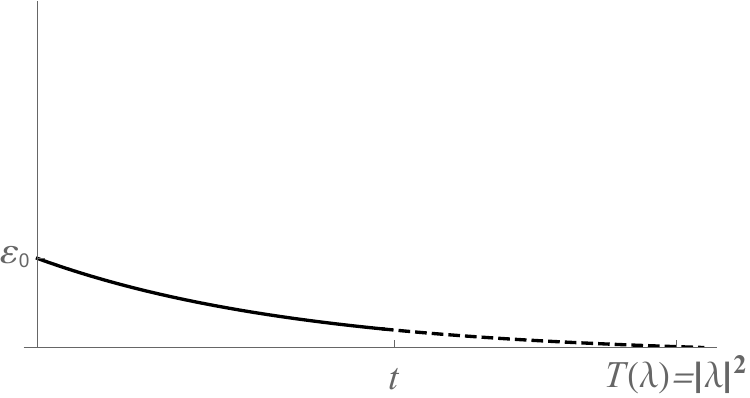}%
\caption{If $\varepsilon_{0}$ is small and positive, $\varepsilon(s)$ will
remain small and positive up to time $t$, \textit{provided} that $t\leq
T(\lambda)=\left\vert \lambda\right\vert ^{2}$ }%
\label{outside.fig}%
\end{figure}

In the case that $\lambda$ is outside the disk, we may then simply let
$\varepsilon_{0}$ approach zero in the Hamilton--Jacobi formula, giving the
following result.

\begin{proposition}
Suppose $\left\vert \lambda\right\vert \geq\sqrt{t},$ that is, that $\lambda$
is outside the open disk of radius $\sqrt{t}$ centered at 0. Then we may let
$\varepsilon_{0}$ tend to zero in the Hamilton--Jacobi formula
(\ref{HJcircular}) to obtain%
\begin{align}
\lim_{\varepsilon\rightarrow0^{+}}S^{\lambda}(t,\varepsilon) &  =\lim
_{\varepsilon_{0}\rightarrow0}\left(  \log(\left\vert \lambda\right\vert
^{2}+\varepsilon_{0})-\frac{\varepsilon_{0}t}{(\left\vert \lambda\right\vert
^{2}+\varepsilon_{0})^{2}}\right)  \nonumber\\
&  =\log(\left\vert \lambda\right\vert ^{2}).\label{sLimOut}%
\end{align}
Since the right-hand side of (\ref{sLimOut}) is harmonic, we conclude that%
\[
\Delta s_{t}(\lambda)=0,\quad\left\vert \lambda\right\vert >\sqrt{t}.
\]
That is to say, the Brown measure of $c_{t}$ is zero outside the disk of
radius $\sqrt{t}$ centered at 0.
\end{proposition}

\subsection{Letting $\varepsilon$ tend to zero: inside the disk}

We now turn to the case in which the time $t$ we are interested in is greater
than the small-$\varepsilon_{0}$ lifetime $T(\lambda)$ of the solutions to
(\ref{Ham1})--(\ref{Ham2}). This case corresponds to $t>T(\lambda
)^{2}=\left\vert \lambda\right\vert ^{2},$ that is, $\left\vert \lambda
\right\vert <\sqrt{t}.$ We still want to choose $\varepsilon_{0}$ so that
$\varepsilon(t)$ will approach zero, but we cannot let $\varepsilon_{0}$ tend
to zero, or else the lifetime of the solution will be less than $t.$ Instead,
we allow the \textit{second} factor in the formula (\ref{xOfT}) for
$\varepsilon(t)$ to approach zero. To make this factor approach zero, we make
$\left\vert \lambda\right\vert ^{2}+\varepsilon_{0}$ approach $t,$ that is,
$\varepsilon_{0}$ should approach $t-\left\vert \lambda\right\vert ^{2}.$ Note
that since we are now assuming that $\left\vert \lambda\right\vert <\sqrt{t},$
the quantity $t-\left\vert \lambda\right\vert ^{2}$ is positive. This strategy
is illustrated in Figure \ref{inside.fig}: When $\varepsilon_{0}=t-\left\vert
\lambda\right\vert ^{2},$ we obtain $\varepsilon(t)=0$ and if $\varepsilon
_{0}$ approaches $t-\left\vert \lambda\right\vert ^{2}$ from above, the value
of $\varepsilon(t)$ approaches 0 from above.%

\begin{figure}[ptb]%
\centering
\includegraphics[
height=2.3506in,
width=3.7775in
]%
{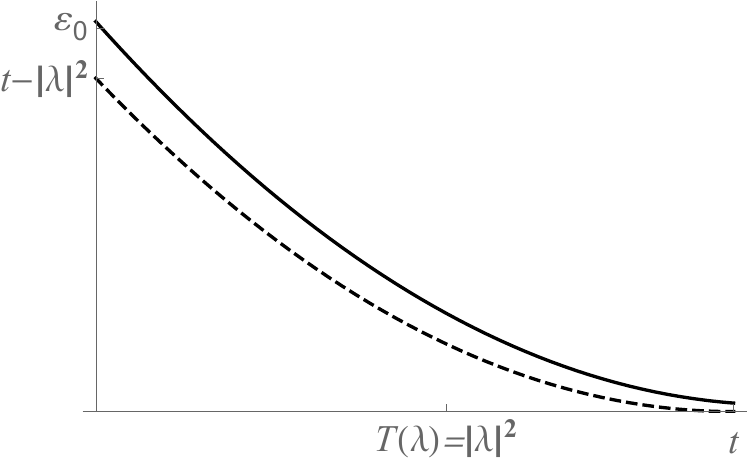}%
\caption{If $\left\vert \lambda\right\vert <\sqrt{t}$ and we let
$\varepsilon_{0}$ approach $t-\left\vert \lambda\right\vert ^{2}$ from above,
$\varepsilon(s)$ will remain positive until time $t$ and $\varepsilon(t)$ will
approach zero}%
\label{inside.fig}%
\end{figure}

\begin{proposition}
Suppose $\left\vert \lambda\right\vert \leq\sqrt{t},$ that is, that $\lambda$
is inside the closed disk of radius $\sqrt{t}$ centered at 0. Then in the
Hamilton--Jacobi formula (\ref{HJcircular}), we may let $\varepsilon_{0}$
approach $t-\left\vert \lambda\right\vert ^{2}$ from above, and we get%
\[
\lim_{\varepsilon\rightarrow0^{+}}S^{\lambda}(t,\varepsilon)=\log
t-1+\frac{\left\vert \lambda\right\vert ^{2}}{t},\quad\left\vert
\lambda\right\vert \leq\sqrt{t}.
\]
For $\left\vert \lambda\right\vert <\sqrt{t}$ we may then compute
\[
\frac{1}{4\pi}\Delta s_{t}(\lambda)=\frac{1}{\pi t}.
\]
Thus, inside the disk of radius $\sqrt{t},$ the Brown measure has a constant
density of $1/(\pi t).$
\end{proposition}

\begin{proof}
We use the Hamilton--Jacobi formula (\ref{HJcircular}). Since the lifetime of
our solution is $\left\vert \lambda\right\vert ^{2}+\varepsilon_{0},$ if we
let $\varepsilon_{0}$ approach $t-\left\vert \lambda\right\vert ^{2}$ from
above, the lifetime will always be at least $t.$ In this limit, the formula
(\ref{xOfT}) for $\varepsilon(t)$ approaches zero from above. Thus, we may
take the limit $\varepsilon_{0}\rightarrow(t-\left\vert \lambda\right\vert
^{2})^{+}$ in (\ref{HJcircular}) to obtain
\begin{align*}
\lim_{\varepsilon\rightarrow0^{+}}S^{\lambda}(t,\varepsilon) &  =\lim
_{\varepsilon_{0}\rightarrow(t-\left\vert \lambda\right\vert ^{2})^{+}}\left[
\log(\left\vert \lambda\right\vert ^{2}+\varepsilon_{0})-\frac{\varepsilon
_{0}t}{(\left\vert \lambda\right\vert ^{2}+\varepsilon_{0})^{2}}\right]  \\
&  =\log t-\frac{(t-\left\vert \lambda\right\vert ^{2})t}{t^{2}},
\end{align*}
which simplifies to the claimed formula.
\end{proof}

\subsection{On the boundary}

Note that if $\left\vert \lambda\right\vert ^{2}=t,$ both approaches are
valid---and the two values of $s_{t}(\lambda):=\lim_{\varepsilon
\rightarrow0^{+}}S^{\lambda}(t,\varepsilon)$ agree, with a common value of
$\log t=\log\left\vert \lambda\right\vert ^{2}.$ Furthermore, the radial
derivatives of $s_{t}(\lambda)$ agree on the boundary: $2/r$ on the outside
and $2r/t$ on the inside, which have a common value of $2/\sqrt{t}$ at
$r=\sqrt{t}.$ Of course, the angular derivatives of $s_{t}(\lambda)$ are
identically zero, inside, outside, and on the boundary.

Since the first derivatives of $s_{t}$ are continuous up to the boundary, we
may take the distributional Laplacian by taking the ordinary Laplacian inside
the disk and outside the disk and ignoring the boundary. (See the proof of
Proposition 7.13 in \cite{DHKBrown}.) Thus, we may compute the Laplacian of
the two formulas in (\ref{stFormula}) to obtain the formula (\ref{WtFormula})
for the Brown measure of $c_{t}.$

\section{The case of the free multiplicative Brownian motion}

\subsection{Additive and multiplicative models}

The standard GUE and Ginibre ensembles are given by Gaussian measures on the
relevant space of matrices (Hermitian matrices for GUE\ and all matrices for
the Ginibre ensemble). In light of the central limit theorem, these ensembles
can be approximated by adding together large numbers of small, independent
random matrices. We may therefore refer to these Gaussian ensembles as
\textquotedblleft additive\textquotedblright\ models.

It is natural to consider also \textquotedblleft
multiplicative\textquotedblright\ random matrix models, which can be
approximated by \textit{multiplying} together large numbers of independent
matrices that are \textquotedblleft small\textquotedblright\ in the
multiplicative sense, that is, close to the identity. Specifically, if
$Z^{\mathrm{add}}$ is a random matrix with a Gaussian distribution, we will
consider a multiplicative version $Z_{t}^{\mathrm{mult}},$ where the
distribution of $Z_{t}^{\mathrm{mult}}$ may be approximated as
\begin{equation}
Z_{t}^{\mathrm{mult}}\sim\prod_{j=1}^{k}\left(  I+i\sqrt{\frac{t}{k}}%
Z_{j}^{\mathrm{add}}-\frac{t}{k}\mathrm{It\hat{o}}\right)  ,\quad k\text{
large.} \label{productK}%
\end{equation}
Here $t$ is a positive parameter, the $Z_{j}^{\mathrm{add}}$'s are independent
copies of $Z^{\mathrm{add}},$ and \textquotedblleft It\^{o}\textquotedblright%
\ is an It\^{o} correction term. This correction term is a fixed multiple of
the identity, independent of $t$ and $k.$ (In the next paragraph, we will
identify the It\^{o} term in the main cases of interest.) Since the factors in
(\ref{productK}) are independent and identically distributed, the order of the
factors does not affect the distribution of the product.

The two main cases we will consider are those in which $Z$ is distributed
according to the Gaussian unitary ensemble or the Ginibre ensemble. In the
case that $Z$ is distributed according to the Gaussian unitary ensemble, the
It\^{o} term is $\mathrm{It\hat{o}}=\frac{1}{2}I.$ In this case, the resulting
multiplicative model may be described as \textit{Brownian motion in the
unitary group} $\mathsf{U}(N),$ which we write as $U_{t}^{N}.$ The It\^{o}
correction is essential in this case to ensure that $Z_{t}^{\mathrm{mult}}$
actually lives in the unitary group. In the case that $Z$ is distributed
according to the Ginibre ensemble, the It\^{o} term is zero. In this case, the
resulting multiplicative model may be described as \textit{Brownian motion in
the general linear group} $\mathsf{GL}(N;\mathbb{C})$, which we write as
$B_{t}^{N}.$

\subsection{The free unitary and free multiplicative Brownian
motions\label{freeBr.sec}}

The large-$N$ limits of the Brownian motions $U_{t}^{N}$ and $B_{t}^{N}$ were
constructed by Biane \cite{BianeFields}. The limits are the \textbf{free
unitary Brownian motion} and the \textbf{free multiplicative Brownian motion},
respectively, which we write as $u_{t}$ and $b_{t}.$ The qualifier
\textquotedblleft free\textquotedblright\ indicates that the increments of
these Brownian motions---computed in the multiplicative sense as $u_{s}%
^{-1}u_{t}$ or $b_{s}^{-1}b_{t}$---are freely independent in the sense of
Section \ref{free.sec}. In the case of $b_{t},$ the convergence of $B_{t}^{N}$
to $b_{t}$ was conjectured by Biane \cite{BianeFields} and proved by Kemp
\cite{KempLargeN}. In both cases, we take the limiting object to be an element
of a tracial von Neumann algebra $(\mathcal{A},\tau).$

Since $u_{t}$ is unitary, we do not need to use the machinery of Brown
measure, but can rather use the spectral theorem as in (\ref{spectral}) to
compute the \textbf{distribution}\ of $u_{t},$ denoted $\nu_{t}.$ We emphasize
that $\nu_{t}$ is, in fact, the Brown measure of $u_{t},$ but it easier to
describe $\nu_{t}$ using the spectral theorem than to use the general Brown
measure construction. The measure $\nu_{t}$ is a probability measure on the
unit circle describing the large-$N$ limit of Brownian motion in the unitary
group $\mathsf{U}(N).$ Biane computed the measure $\nu_{t}$ in
\cite{BianeFields} and established the following support result.

\begin{theorem}
\label{bianeSupport.thm}For $t<4,$ the measure $\nu_{t}$ is supported on a
proper subset of the unit circle:%
\[
\mathrm{supp}(\nu_{t})=\left\{  \left.  e^{i\theta}\right\vert ~\left\vert
\theta\right\vert \leq\frac{1}{2}\sqrt{t(4-t)}+\cos^{-1}\left(  1-\frac{t}%
{2}\right)  \right\}  ,\quad t<4.
\]
By contrast, for all $t\geq4,$ the closed support of $\nu_{t}$ is the whole
unit circle.
\end{theorem}

In the physics literature, the change in behavior of the support of $\nu_{t}$
at $t=4$ is called a \textit{topological phase transition}, indicating that
the topology of $\mathrm{supp}(\nu_{t})$ changes from a closed interval to a circle.%

\begin{figure}[ptb]%
\centering
\includegraphics[
height=2.4734in,
width=2.5278in
]%
{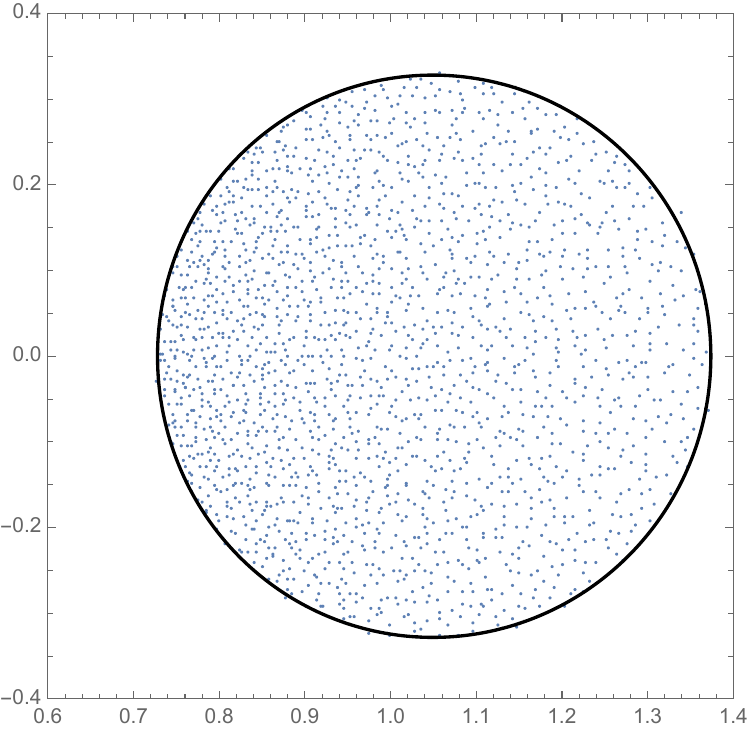}%
\caption{The eigenvalues of $B_{t}^{N}$ with $t=0.1$ and $N=2.000$}%
\label{t01.fig}%
\end{figure}

The remainder of this article is devoted to recent results of the author with
Driver and Kemp regarding the Brown measure of the free multiplicative
Brownian motion $b_{t}.$ We expect that the Brown measure of $b_{t}$ will be
the limiting empirical eigenvalue distribution of the Brownian motion
$B_{t}^{N}$ in the general linear group $\mathsf{GL}(N;\mathbb{C})$. Now, when
$t$ is small, we may take $k=1$ in (\ref{productK}), so that (since the
It\^{o} correction is zero in this case),
\[
B_{t}^{N}\sim I+i\sqrt{\frac{t}{k}}Z,\quad t\text{ small.}%
\]
Thus, when $t$ is small and $N$ is large, the eigenvalues of $B_{t}^{N}$
resemble a scaled and shifted version of the circular law. Specifically, the
eigenvalue distribution should resemble a uniform distribution on the disk of
radius $\sqrt{t}$ centered at 1.%

\begin{figure}[ptb]%
\centering
\includegraphics[
height=2.1179in,
width=4.0283in
]%
{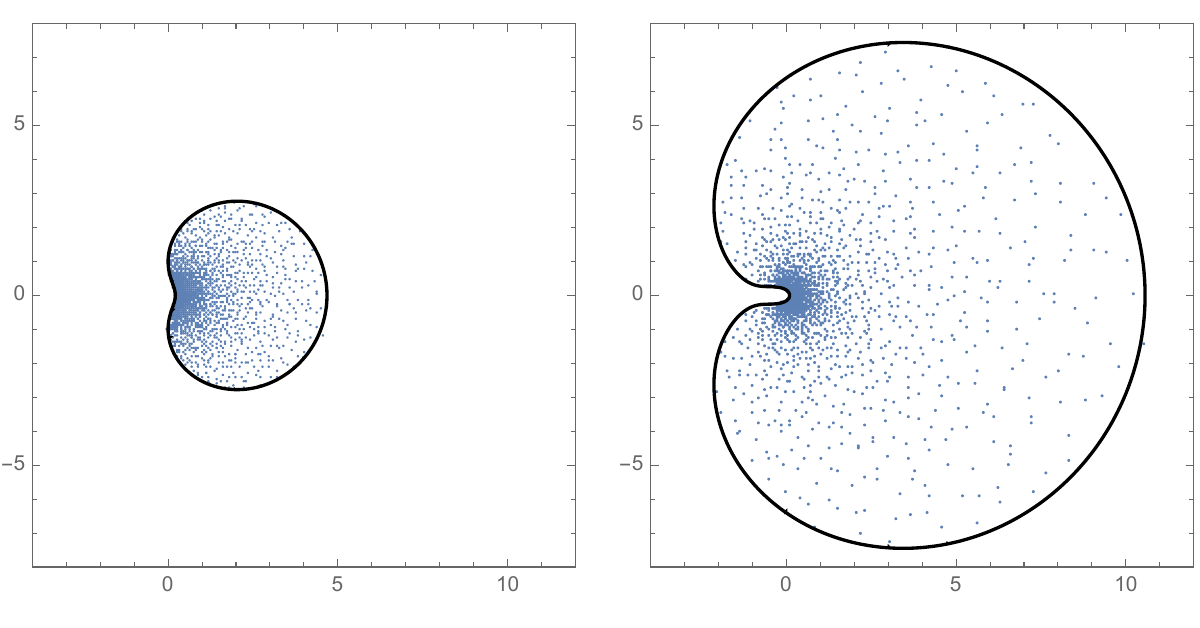}%
\caption{Eigenvalues of $B_{t}^{N}$ for $t=2$ (left) and $t=3.9$ (right), with
$N=2,000$}%
\label{t2and39.fig}%
\end{figure}

Figure \ref{t01.fig} shows the eigenvalues of $B_{t}^{N}$ with $t=0.1$ and
$N=2,000.$ The eigenvalue distribution bears a clear resemblance to the
just-described picture, with $\sqrt{t}=\sqrt{0.1}\approx0.316.$ Nevertheless,
we can already see some deviation from the small-$t$ picture: The region into
which the eigenvalues are clustering looks like a disk, but not quite centered
at 1, while the distribution within the region is slightly higher at the
left-hand side of the region than the right. Figures \ref{t2and39.fig} and
\ref{t4and41.fig}, meanwhile, show the eigenvalue distribution of $B_{t}^{N}$
for several larger values of $t.$ The region into which the eigenvalues
cluster becomes more complicated as $t$ increases, and the distribution of
eigenvalues in the region becomes less and less uniform. We expect that the
Brown measure of the limiting object $b_{t}$ will be supported on the domain
into which the eigenvalues are clustering.

%

\begin{figure}[ptb]%
\centering
\includegraphics[
height=2.1179in,
width=4.0283in
]%
{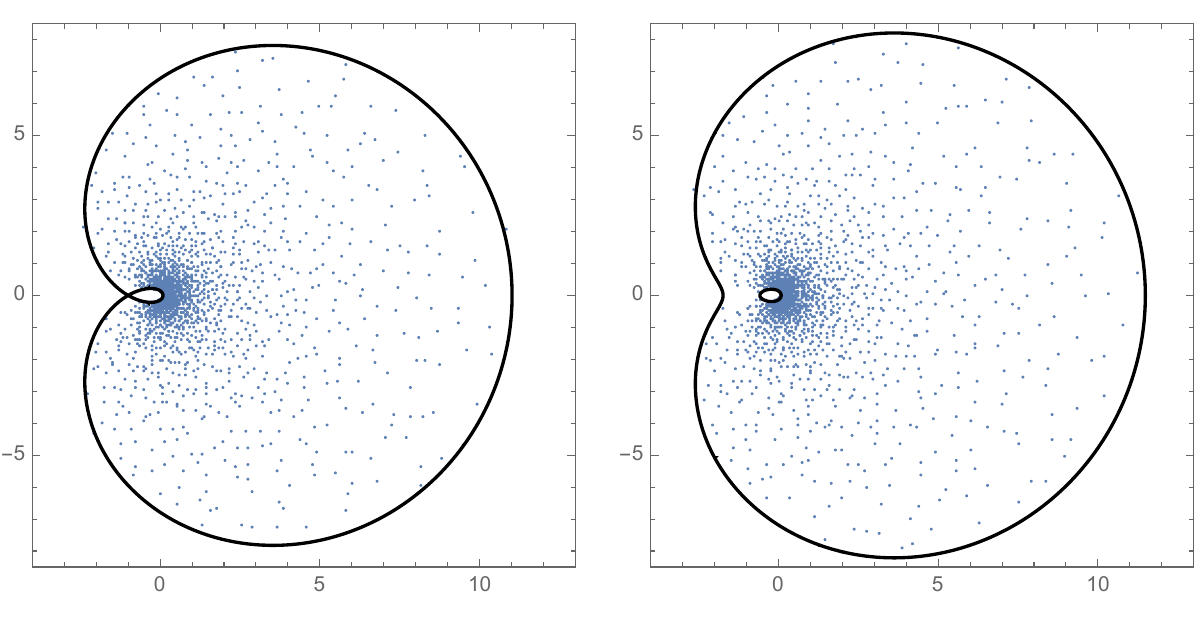}%
\caption{Eigenvalues of $B_{t}^{N}$ for $t=4$ (left) and $t=4.1$ (right), with
$N=2,000$}%
\label{t4and41.fig}%
\end{figure}

\subsection{The domains $\Sigma_{t}$\label{domains.sec}}%

\begin{figure}[ptb]%
\centering
\includegraphics[
height=3.0277in,
width=2.3402in
]%
{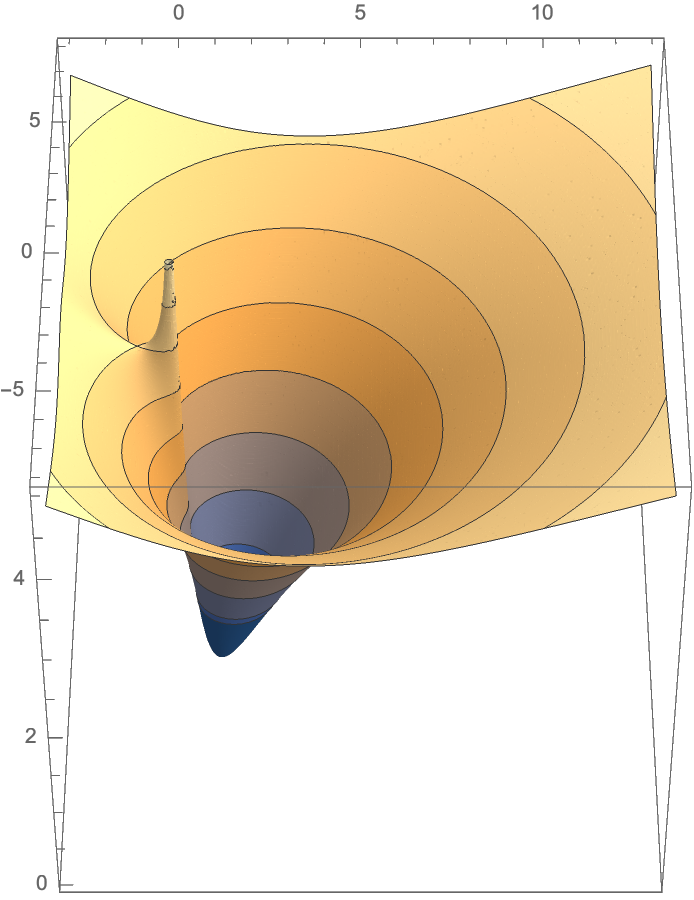}%
\caption{A plot of the function $T(\lambda).$ The function has a minimum at
$\lambda=1,$ a saddle point at $\lambda=-1,$ and a singularity at $\lambda
=0.$}%
\label{tplot.fig}%
\end{figure}

We now describe certain domains $\Sigma_{t}$ in the plane, as introduced by
Biane in \cite[pp. 273-274]{BianeJFA}. It will turn out that the Brown measure
of $b_{t}$ is supported on $\Sigma_{t}.$ We use here a new the description of
$\Sigma_{t},$ as given in Section 4 of \cite{DHKBrown}. For all nonzero
$\lambda\in\mathbb{C},$ we define
\begin{equation}
T(\lambda)=\left\vert \lambda-1\right\vert ^{2}\frac{\log(\left\vert
\lambda\right\vert ^{2})}{\left\vert \lambda\right\vert ^{2}-1}. \label{Tdef}%
\end{equation}
If $\left\vert \lambda\right\vert ^{2}=1,$ we interpret $\log(\left\vert
\lambda\right\vert ^{2})/(\left\vert \lambda\right\vert ^{2}-1)$ as having the
value 1 when $\left\vert \lambda\right\vert ^{2}=1,$ in accordance with the
limit
\[
\lim_{r\rightarrow1}\frac{\log r}{r-1}=1.
\]
See Figure \ref{tplot.fig} for a plot of this function.

We then define the domains $\Sigma_{t}$ as follows.

\begin{definition}
\label{sigma.def}For each $t>0,$ we define%
\[
\Sigma_{t}=\left\{  \left.  \lambda\in\mathbb{C}\right\vert T(\lambda
)<t\right\}  .
\]

\end{definition}

Several examples of these domains were plotted already in Figures
\ref{t01.fig}, \ref{t2and39.fig}, and \ref{t4and41.fig}. The domain
$\Sigma_{t}$ is simply connected for $t\leq4$ and doubly connected for $t>4.$
The change in behavior at $t=4$ occurs because $T$ has a saddle point at
$\lambda=-1$ and because $T(-1)=4.$ We note that a change in the topology of
the region occurs at $t=4,$ which is the same value of $t$ at which the
topology of the support of Biane's measure changes (Theorem
\ref{bianeSupport.thm}).

\subsection{The support of the Brown measure of $b_{t}$}

As we have noted, the domains $\Sigma_{t}$ were introduced by Biane in
\cite{BianeJFA}. Two subsequent works in the physics literature, the article
\cite{Nowak} by Gudowska-Nowak, Janik, Jurkiewicz, and Nowak and the article
\cite{Lohmayer} by Lohmayer, Neuberger, and Wettig then argued, using
nonrigorous methods, that the eigenvalues of $B_{t}^{N}$ should concentrate
into $\Sigma_{t}$ for large $N.$ The first rigorous result in this direction
was obtained by the author with Kemp \cite{HK}; we prove that the Brown
measure of $b_{t}$ is supported on the closure of $\Sigma_{t}.$

Now, we have already noted that $\Sigma_{t}$ is simply connected for $t\leq4$
but doubly connected for $t>4.$ Thus, the support of the Brown measure of the
free \textit{multiplicative} Brownian motion undergoes a \textquotedblleft
topological phase transition\textquotedblright\ at precisely the same value of
the time-parameter as the distribution of the free \textit{unitary} Brownian
motion (Theorem \ref{bianeSupport.thm}).

The methods of \cite{HK} explain this apparent coincidence, using the
\textquotedblleft free Hall transform\textquotedblright\ $\mathcal{G}_{t}$\ of
Biane \cite{BianeJFA}. Biane constructed this transform using methods of free
probability as an infinite-dimensional analog of the Segal--Bargmann transform
for $\mathsf{U}(N),$ which was developed by the author in \cite{Ha1994}. More
specifically, Biane's definition $\mathcal{G}_{t}$ draws on the stochastic
interpretation of the transform in \cite{Ha1994} given by Gross and Malliavin
\cite{GM}. Biane conjectured (with an outline of a proof) that $\mathcal{G}%
_{t}$ is actually the large-$N$ limit of the transform in \cite{Ha1994}. This
conjecture was then verified by in independent works of C\'{e}bron \cite{Ceb}
and the author with Driver and Kemp \cite{DHKLargeN}. (See also the expository
article \cite{HallExpository}.)

Recall from Section \ref{freeBr.sec} that the distribution of the free unitary
Brownian motion is Biane's measure $\nu_{t}$ on the unit circle, the support
of which is described in Theorem \ref{bianeSupport.thm}. A key ingredient in
\cite{HK} is the function $f_{t}$ given by%
\begin{equation}
f_{t}(\lambda)=\lambda e^{\frac{t}{2}\frac{1+\lambda}{1-\lambda}}. \label{ft1}%
\end{equation}
This function maps the complement of the closure of $\Sigma_{t}$ conformally
to the complement of the support of Biane's measure:%
\begin{equation}
f_{t}:\mathbb{C}\setminus\overline{\Sigma}_{t}\rightarrow\mathbb{C}%
\setminus\mathrm{supp}(\nu_{t}). \label{ft2}%
\end{equation}
(This map $f_{t}$ will also play a role in the results of Section
\ref{BrownBt.sec}; see Theorem \ref{connectToBiane.thm}.)

The key computation in \cite{HK} is that for $\lambda$ outside $\overline
{\Sigma}_{t},$ we have%
\begin{equation}
\mathcal{G}_{t}^{-1}\left(  \frac{1}{z-\lambda}\right)  =\frac{f_{t}(\lambda
)}{\lambda}\frac{1}{u-f_{t}(\lambda)},\quad\lambda\notin\overline{\Sigma}_{t}.
\label{GtInv}%
\end{equation}
See Theorem 6.8 in \cite{HK}. Properties of the free Hall transform then imply
that for $\lambda$ outside $\overline{\Sigma}_{t},$ the operator
$b_{t}-\lambda$ has an inverse. Indeed, the noncommutative $L^{2}$ norm of
$(b_{t}-\lambda)^{-1}$ equals to the norm in $L^{2}(S^{1},\nu_{t})$ of the
function on the right-hand side of (\ref{GtInv}). This norm, in turn, is
finite because $f_{t}(\lambda)$ is outside the support of $\nu_{t}$ whenever
$\lambda$ is outside $\overline{\Sigma}_{t}.$ The existence of an inverse to
$b_{t}-\lambda$ then shows that $\lambda$ must be outside the support of
$\mu_{b_{t}}.$

An interesting aspect of the paper \cite{HK} is that we not only
\textit{compute} the support of $\mu_{b_{t}},$ but also that we
\textit{connect} it to the support of Biane's measure $\nu_{t},$ using the
transform $\mathcal{G}_{t}$ and the conformal map $f_{t}.$

We note, however, that none of the papers \cite{Nowak}, \cite{Lohmayer}, or
\cite{HK} says anything about the distribution of $\mu_{b_{t}}$ within
$\Sigma_{t}$; they are only concerned with identifying the region $\Sigma
_{t}.$ The actual computation of $\mu_{b_{t}}$ (not just its support) was done
in \cite{DHKBrown}.

\subsection{The Brown measure of $b_{t}$\label{BrownBt.sec}}

We now describe the main results of \cite{DHKBrown}. Many of these results
have been extended by Ho and Zhong \cite{HZ} to the case of the free
multiplicative Brownian motion with an arbitrary unitary initial distribution.

The first key result in \cite{DHKBrown} is the following formula for the Brown
measure of $b_{t}$ (Theorem 2.2 of \cite{DHKBrown}).

\begin{theorem}
\label{dhkMain.thm}For each $t>0,$ the Brown measure $\mu_{b_{t}}$ is zero
outside the closure of the region $\Sigma_{t}.$ In the region $\Sigma_{t},$
the Brown measure has a density $W_{t}$ with respect to Lebesgue measure. This
density has the following special form in polar coordinates:%
\[
W_{t}(r,\theta)=\frac{1}{r^{2}}w_{t}(\theta),\quad re^{i\theta}\in\Sigma_{t},
\]
for some positive continuous function $w_{t}.$ The function $w_{t}$ is
determined entirely by the geometry of the domain and is given as%
\[
w_{t}(\theta)=\frac{1}{4\pi}\left(  \frac{2}{t}+\frac{\partial}{\partial
\theta}\frac{2r_{t}(\theta)\sin\theta}{r_{t}(\theta)^{2}+1-2r_{t}(\theta
)\cos\theta}\right)  ,
\]
where $r_{t}(\theta)$ is the \textquotedblleft outer radius\textquotedblright%
\ of the region $\Sigma_{t}$ at angle $\theta.$
\end{theorem}

See Figure \ref{rt.fig} for the definition of $r_{t}(\theta),$ Figure
\ref{wtplots.fig} for plots of the function $w_{t}(\theta)$, and Figure
\ref{w3d.fig} for a plot of $W_{t}.$ The simple explicit dependence of $W_{t}$
on $r$ is a major surprise of our analysis. See Corollary
\ref{logDistribution.cor} for a notable consequence of the form of $W_{t}.$

Using implicit differentiation, it is possible to compute $dr_{t}%
(\theta)/d\theta$ explicitly as a function of $r_{t}(\theta).$ This
computation yields the following formula for $w_{t},$ which does not involve
differentiation:%
\[
w_{t}(\theta)=\frac{1}{2\pi t}\omega(r_{t}(\theta),\theta),
\]
where%
\begin{equation}
\omega(r,\theta)=1+h(r)\frac{\alpha(r)\cos\theta+\beta(r)}{\beta(r)\cos
\theta+\alpha(r)}, \label{omegaFormula}%
\end{equation}
and%
\[
h(r)=r\frac{\log(r^{2})}{r^{2}-1};\quad\alpha(r)=r^{2}+1-2rh(r);\quad
\beta(r)=(r^{2}+1)h(r)-2r.
\]
See Proposition 2.3 in \cite{DHKBrown}.%

\begin{figure}[ptb]%
\centering
\includegraphics[
height=2.5278in,
width=2.5278in
]%
{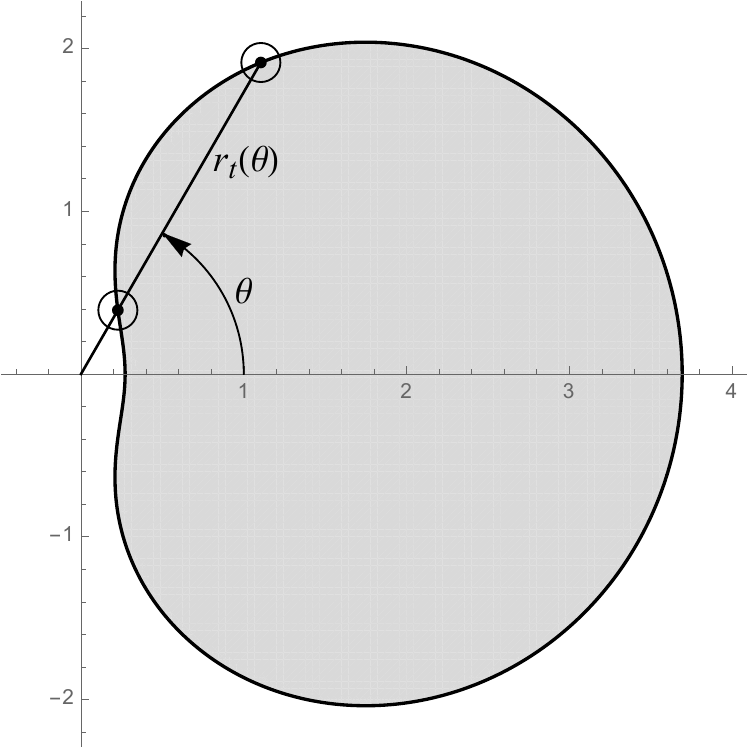}%
\caption{The quantity $r_{t}(\theta)$ is the larger of the two radii at which
the ray of angle $\theta$ intersects the boundary of $\Sigma_{t}$}%
\label{rt.fig}%
\end{figure}
%

\begin{figure}[ptb]%
\centering
\includegraphics[
height=2.5002in,
width=4.0283in
]%
{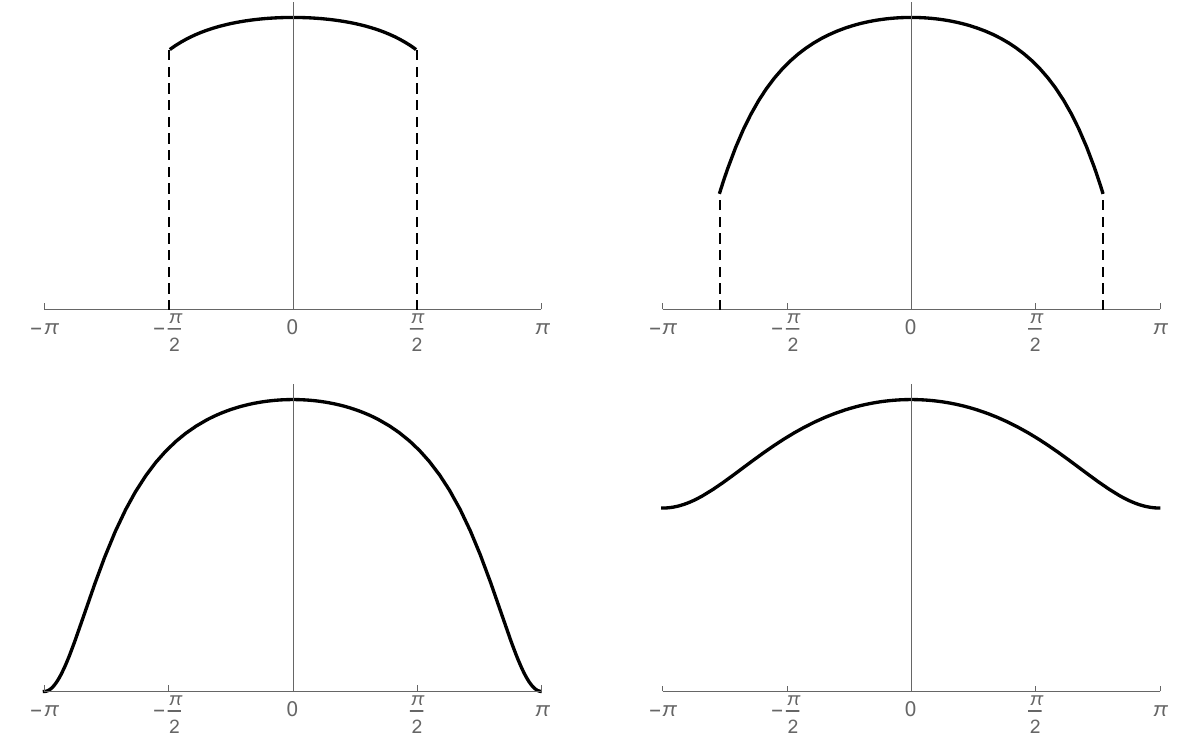}%
\caption{Plots of $w_{t}(\theta)$ for $t=2,$ $3.5,$ $4,$ and $7$}%
\label{wtplots.fig}%
\end{figure}
%

\begin{figure}[ptb]%
\centering
\includegraphics[
height=1.8351in,
width=3.0277in
]%
{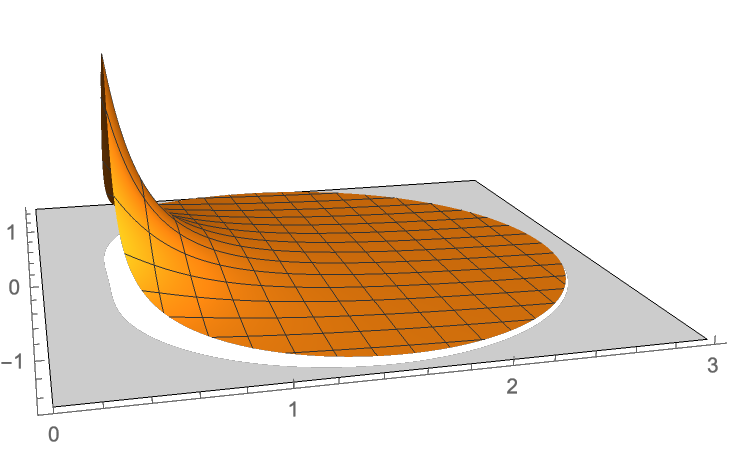}%
\caption{Plot of the density $W_{t}$ for $t=1$}%
\label{w3d.fig}%
\end{figure}

We expect that the Brown measure of $b_{t}$ will coincide with the limiting
empirical eigenvalue distribution of the Brownian motion $B_{t}^{N}$ in
$\mathsf{GL}(N;\mathbb{C}).$ This expectation is supported by simulations; see
Figure \ref{3dplotwithhist.fig}.%

\begin{figure}[ptb]%
\centering
\includegraphics[
height=1.9112in,
width=4.3863in
]%
{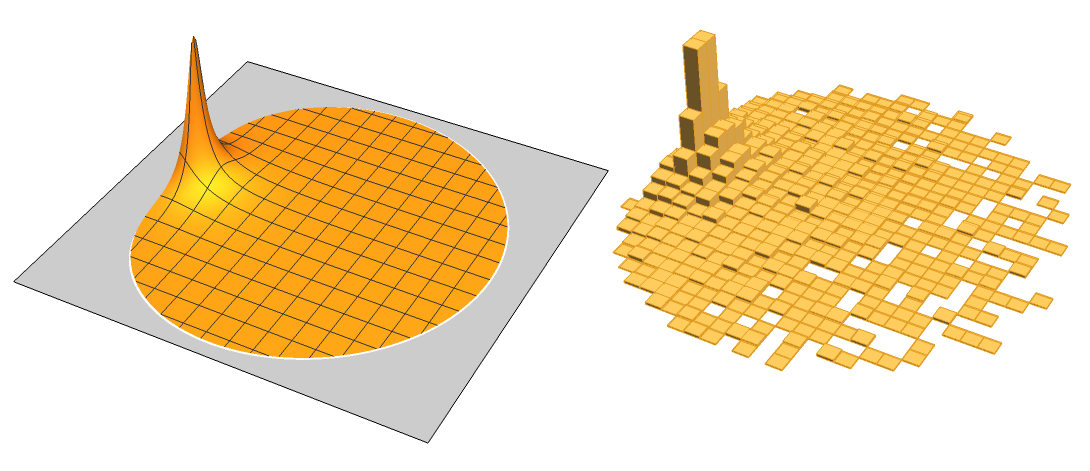}%
\caption{The density $W_{t}$ (left) and a histogram of the eigenvalues of
$B_{t}^{N}$ (right), for $t=1$ and $N=2,000$}%
\label{3dplotwithhist.fig}%
\end{figure}

We note that the Brown measure (inside $\Sigma_{t}$) can also be written as
\begin{align*}
d\mu_{b_{t}}  &  =\frac{1}{r^{2}}w_{t}(\theta)~r~dr~d\theta\\
&  =w_{t}(\theta)~\frac{1}{r}~dr~d\theta\\
&  =w_{t}(\theta)~d\log r~d\theta.
\end{align*}
Since the complex logarithm is given by $\log(re^{i\theta})=\log r+i\theta,$
we obtain the following consequence of Theorem \ref{dhkMain.thm}.

\begin{corollary}
\label{logDistribution.cor}The push-forward of the Brown measure $\mu_{b_{t}}$
under the complex logarithm has density that is constant in the horizontal
direction and given by $w_{t}$ in the vertical direction.
\end{corollary}

In light of this corollary, we expect that for large $N,$ the logarithms of
the eigenvalues of $B_{t}^{N}$ should be approximately uniformly distributed
in the horizontal direction. This expectation is confirmed by simulations, as
in Figure \ref{evalsandlogs.fig}.%

\begin{figure}[ptb]%
\centering
\includegraphics[
height=2.2027in,
width=4.1943in
]%
{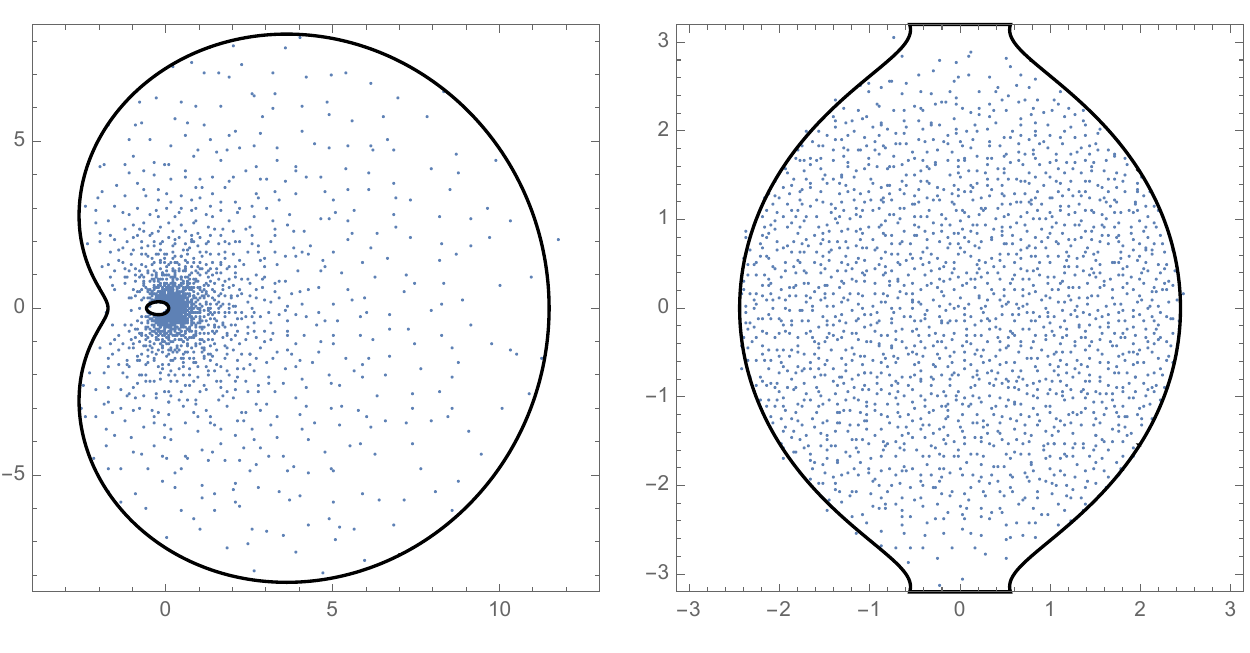}%
\caption{The eigenvalues of $B_{t}^{N}$ for $t=4.1$ and $N=2,000$ (left) and
the logarithms thereof (right). The density of points on the right-hand side
of the figure is approximately constant in the horizontal direction}%
\label{evalsandlogs.fig}%
\end{figure}

We conclude this section by describing a remarkable connection between the
Brown measure $\mu_{b_{t}}$ and the distribution $\nu_{t}$ of the free unitary
Brownian motion. Recall the holomorphic function $f_{t}$ in (\ref{ft1})\ and
(\ref{ft2}). This map takes the boundary of $\Sigma_{t}$ to the unit circle.
We ma then define a map
\[
\Phi_{t}:\overline{\Sigma}_{t}\rightarrow S^{1}%
\]
by requiring (a) that $\Phi_{t}$ should agree with $f_{t}$ on the boundary of
$\Sigma_{t},$ and (b) that $\Phi_{t}$ should be constant along each radial
segment inside $\overline{\Sigma}_{t},$ as in Figure \ref{biane.fig}. (This
specification makes sense because $f_{t}$ has the same value at the two
boundary points on each radial segment.) We then have the following result,
which may be summarized by saying that \emph{the distribution $\nu_{t}$ of
free unitary Brownian motion is a \textquotedblleft shadow\textquotedblright%
\ of the Brown measure of $b_{t}$}.%

\begin{figure}[ptb]%
\centering
\includegraphics[
height=2.1179in,
width=4.0283in
]%
{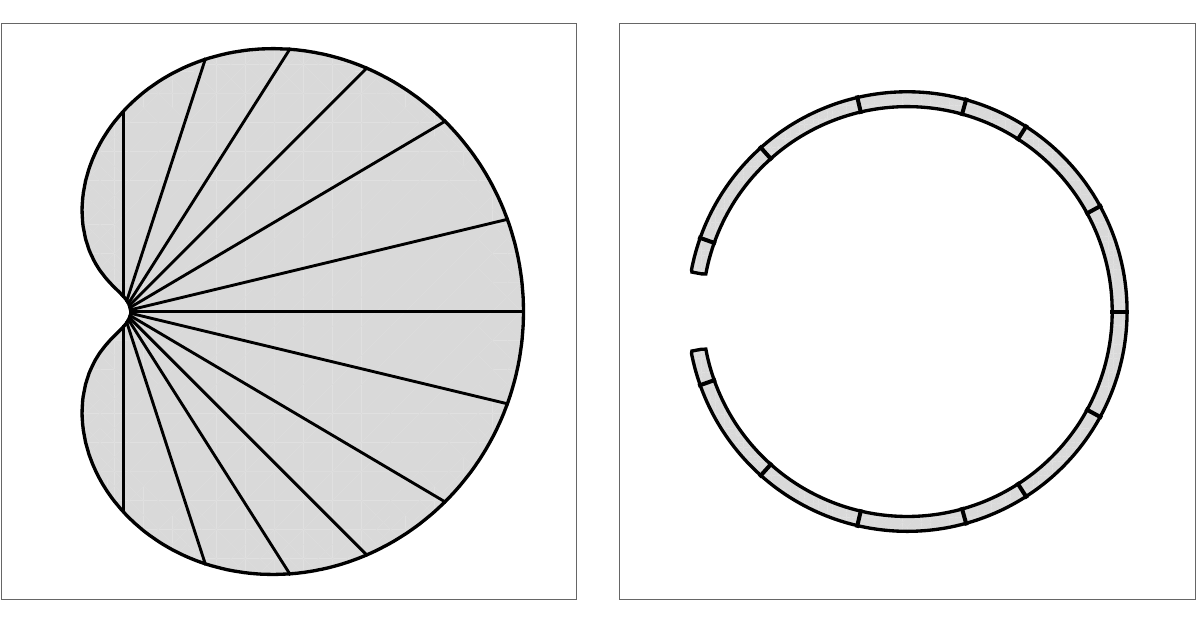}%
\caption{The map $\Phi_{t}$ maps $\overline{\Sigma}_{t}$ to the unit circle by
mapping each radial segment in $\overline{\Sigma}_{t}$ to a single point in
\thinspace$S^{1}$}%
\label{biane.fig}%
\end{figure}

\begin{theorem}
\label{connectToBiane.thm}The push-forward of the Brown measure of $b_{t}$
under the map $\Phi_{t}$ is Biane's measure $\nu_{t}$ on $S^{1}.$ Indeed, the
Brown measure of $b_{t}$ is the \emph{unique} measure $\mu$ on $\overline
{\Sigma}_{t}$ with the following two properties: (1) the push-forward of $\mu$
by $\Phi_{t}$ is $\nu_{t}$ and (2) $\mu$ is absolutely continuous with respect
to Lebesgue measure with a density $W$ having the form%
\[
W(r,\theta)=\frac{1}{r^{2}}g(\theta)
\]
in polar coordinates, for some continuous function $g.$
\end{theorem}

This result is Proposition 2.6 in \cite{DHKBrown}. Figure \ref{bianeevals.fig}
shows the eigenvalues for $B_{t}^{N}$ after applying the map $\Phi_{t},$
plotted against the density of Biane's measure $\nu_{t}.$ We emphasize that we
have computed the eigenvalues of the Brownian $B_{t}^{N}$ motion in
$\mathsf{GL}(N;\mathbb{C})$ (in the two-dimensional region $\Sigma_{t}$) and
then mapped these points to the unit circle. The resulting histogram, however,
looks precisely like a histogram of the eigenvalues of the Brownian motion in
$\mathsf{U}(N).$%

\begin{figure}[ptb]%
\centering
\includegraphics[
height=2.0609in,
width=3.0277in
]%
{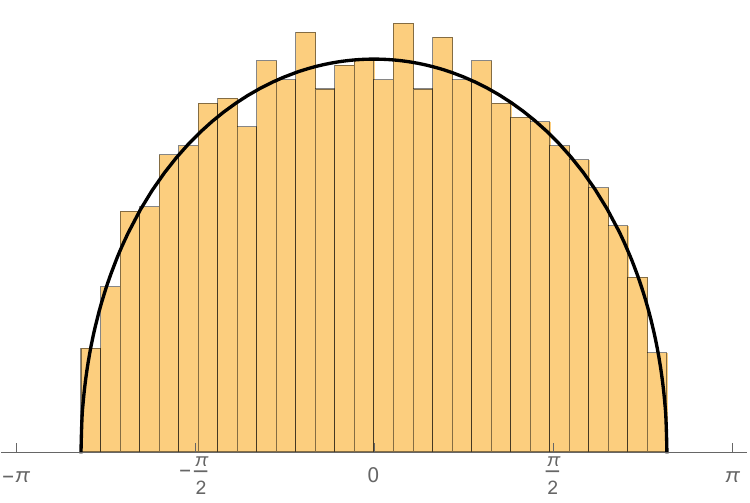}%
\caption{The eigenvalues of $B_{t}^{N}$, mapped to the unit circle by
$\Phi_{t},$ plotted against the density of Biane's measure $\nu_{t}.$ Shown
for $t=2$ and $N=2,000$}%
\label{bianeevals.fig}%
\end{figure}

\subsection{The PDE and its solution}

We conclude this article by briefly outlining the methods used to obtain the
results in the previous subsection.

\subsubsection{The PDE}

Following the definition of the Brown measure in Theorem
\ref{BrownMeasure.thm}, we consider the function
\begin{equation}
S(t,\lambda,\varepsilon):=\tau\lbrack\log((b_{t}-\lambda)^{\ast}(b_{t}%
-\lambda)+\varepsilon)].\label{Smult}%
\end{equation}
We then record the following result \cite[Theorem 2.8]{DHKBrown}.

\begin{theorem}
\label{thePDE.thm}The function $S$ in (\ref{Smult}) satisfies the following
PDE:%
\begin{equation}
\frac{\partial S}{\partial t}=\varepsilon\frac{\partial S}{\partial
\varepsilon}\left(  1+(\left\vert \lambda\right\vert ^{2}-\varepsilon
)\frac{\partial S}{\partial\varepsilon}-a\frac{\partial S}{\partial a}%
-b\frac{\partial S}{\partial b}\right)  ,\quad\lambda=a+ib,\label{thePDE}%
\end{equation}
with the initial condition%
\begin{equation}
S(0,\lambda,\varepsilon)=\log(\left\vert \lambda-1\right\vert ^{2}%
+\varepsilon).\label{SinitialCond}%
\end{equation}

\end{theorem}

Recall that in the case of the circular Brownian motion (the PDE in Theorem
\ref{circularPDE.thm}), the complex number $\lambda$ enters only into the
initial condition and not into the PDE itself. By contrast, the right-hand
side of the PDE (\ref{thePDE}) involves differentiation with respect to the
real and imaginary parts of $\lambda.$

On the other hand, the PDE (\ref{thePDE}) is again of Hamilton--Jacobi type.
Thus, following the general Hamilton--Jacobi method in Section
\ref{hjMethod.sec}, we define a Hamiltonian function $H$ from (the negative
of) the right-hand side of (\ref{thePDE}), replacing each derivative of $S$ by
a corresponding momentum variable:
\begin{equation}
H(a,b,\varepsilon,p_{a},p_{b},p_{\varepsilon})=-\varepsilon p_{\varepsilon
}(1+(a^{2}+b^{2})p_{\varepsilon}-\varepsilon p_{\varepsilon}-ap_{a}%
-bp_{b}).\label{theHamiltonian}%
\end{equation}
We then consider Hamilton's equations for this Hamiltonian:%
\begin{align}
\frac{da}{dt} &  =\frac{\partial H}{\partial p_{a}};\quad~~\frac{db}{dt}%
=\frac{\partial H}{\partial p_{b}};\quad~~~\frac{d\varepsilon}{dt}%
=\frac{\partial H}{\partial p_{\varepsilon}};\nonumber\\
\frac{dp_{a}}{dt} &  =-\frac{\partial H}{\partial a};\quad\frac{dp_{b}}%
{dt}=-\frac{\partial H}{\partial b};\quad\frac{dp_{\varepsilon}}{dt}%
=-\frac{\partial H}{\partial\varepsilon}.\label{theODEs}%
\end{align}
Then, after a bit of simplification, the general Hamilton--Jacobi formula in
(\ref{HJformulaGen}) then takes the form%
\begin{align}
S(t,\lambda(t),\varepsilon(t)) &  =\log(\left\vert \lambda_{0}-1\right\vert
^{2}+\varepsilon_{0})-\frac{\varepsilon_{0}t}{(\left\vert \lambda
_{0}-1\right\vert ^{2}+\varepsilon_{0})^{2}}\nonumber\\
&  +\log\left\vert \lambda(t)\right\vert -\log\left\vert \lambda
_{0}\right\vert \text{.}\label{SmultFormula}%
\end{align}
(See Theorem 6.2 in \cite{DHKBrown}.)

The analysis in \cite{DHKBrown} then proceeds along broadly similar lines to
those in Sections \ref{solving.sec} and \ref{xToZero.sec}. The main structural
difference is that because $\lambda$ is now a variable in the PDE, the ODE's
in (\ref{theODEs}) now involve both $\lambda$ and $\varepsilon,$ and the
associated momenta. (That is to say, the vector $\mathbf{x}$ in Proposition
\ref{HJgeneral.prop} is equal to $(\lambda,\varepsilon)\in\mathbb{C}%
\times\mathbb{R}\cong\mathbb{R}^{3}.$) The first key result is that the system
of ODE's associated to (\ref{thePDE}) can be solved explicitly; see Section
6.3 of \cite{DHKBrown}. Solving the ODE's gives an implicit formula for the
solution to (\ref{thePDE}) with the initial conditions (\ref{SinitialCond}).

We then evaluate the solution in the limit as $\varepsilon$ tends to zero. We
follow the strategy in Section \ref{xToZero.sec}. Given a time $t$ and a
complex number $\lambda,$ we attempt to choose initial conditions
$\varepsilon_{0}$ and $\lambda_{0}$ so that $\varepsilon(t)$ will be very
close to zero and $\lambda(t)$ will equal $\lambda.$ (Recall that the initial
momenta in the system of ODE's are determined by the positions by
(\ref{HJinit}).)

\subsubsection{Outside the domain}

As in the case of the circular Brownian motion, we use different approaches
for $\lambda$ outside $\Sigma_{t}$ and for $\lambda$ in $\Sigma_{t}.$ For
$\lambda$ outside $\Sigma_{t},$ we allow the initial condition $\varepsilon
_{0}$ in the ODE's to approach zero. As it turns out, when $\varepsilon_{0}$
is small and positive, $\varepsilon(t)$ remains small and positive for as long
as the solution to the system exists. Furthermore, when $\varepsilon_{0}$ is
small and positive, $\lambda(t)$ is approximately constant. Thus, our strategy
will be to take $\varepsilon_{0}\approx0$ and $\lambda_{0}\approx\lambda.$

A key result is the following.

\begin{proposition}
\label{lifetime.prop} In the limit as $\varepsilon_{0}$ tends to zero, the
lifetime of the solution to (\ref{theODEs}) with initial conditions
$\lambda_{0}$ and $\varepsilon_{0}$---and initial moment determined by
(\ref{HJinit})---approaches $T(\lambda_{0}),$ where $T$ is the same function
(\ref{Tdef}) that enters into the definition of the domain $\Sigma_{t}.$
\end{proposition}

This result is Proposition 6.13 in \cite{DHKBrown}. Thus, the strategy in the
previous paragraph will work---meaning that the solution continues to exist up
to time $t$---provided that $T(\lambda_{0})\approx T(\lambda)$ is greater than
$t.$ The condition for success of the strategy is, therefore, $T(\lambda)>t.$
In light of the characterization of $\Sigma_{t}$ in Definition \ref{sigma.def}%
, we make have the following conclusion.

\begin{conclusion}
The simple strategy of taking $\varepsilon_{0}\approx0$ and $\lambda
_{0}\approx\lambda$ is successful precisely if $T(\lambda)>t,$ or
equivalently, if $\lambda$ is outside $\overline{\Sigma}_{t}.$
\end{conclusion}

When this strategy works, we obtain a simple expression for $\lim
_{\varepsilon\rightarrow0^{+}}S(t,\lambda,\varepsilon),$ by letting
$\varepsilon_{0}$ approach zero and $\lambda_{0}$ approach $\lambda$ in
(\ref{SmultFormula}). Since $\lambda(t)$ approaches $\lambda$ in this limit
\cite[Proposition 6.11]{DHKBrown}, we find that
\begin{equation}
\lim_{\varepsilon\rightarrow0^{+}}S(t,\lambda,\varepsilon)=\log(\left\vert
\lambda-1\right\vert ^{2}),\quad\lambda\notin\overline{\Sigma}_{t}%
.\label{stOutside}%
\end{equation}
This function is harmonic (except at $\lambda=1,$ which is always in the
domain $\Sigma_{t}$), so we conclude that \textit{the Brown measure of }%
$b_{t}$\textit{ is zero outside} $\overline{\Sigma}_{t}.$ See Section 7.2 in
\cite{DHKBrown} for more details.

\subsubsection{Inside the domain}

For $\lambda$ inside $\Sigma_{t},$ the simple approach in the previous
subsection does not work, because when $\lambda$ is outside $\Sigma_{t}$ and
$\varepsilon_{0}$ is small, the solutions to the ODE's (\ref{theODEs}) will
cease to exist prior to time $t$ (Proposition \ref{lifetime.prop}). Instead,
we must prove a \textquotedblleft surjectivity\textquotedblright\ result: For
each $t>0$ and $\lambda\in\Sigma_{t},$ there exist---in principle---$\lambda
_{0}\in\mathbb{C}$ and $\varepsilon_{0}>0$ giving $\lambda(t)=\lambda$ and
$\varepsilon(t)=0.$ See Figure \ref{airplane.fig}. Actually the proof shows
that $\lambda_{0}$ again belongs to the domain $\Sigma_{t}$; see Section 6.5
in \cite{DHKBrown}.%

\begin{figure}[ptb]%
\centering
\includegraphics[
height=2.2943in,
width=4.0274in
]%
{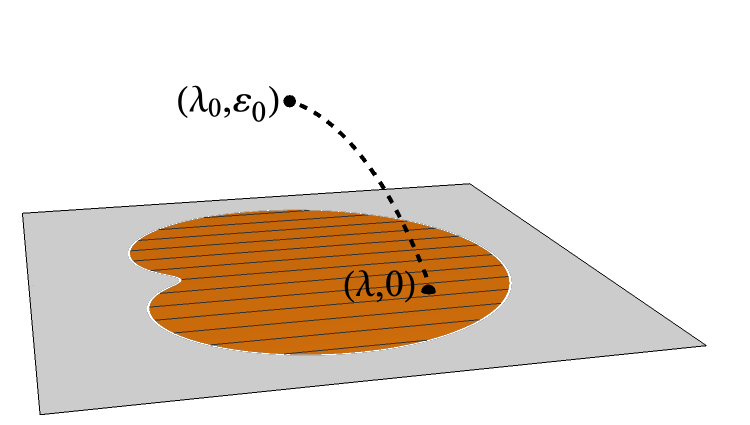}%
\caption{For each $\lambda$ in $\Sigma_{t,}$ there exists $\varepsilon_{0}>0$
and $\lambda_{0}\in\Sigma_{t}$ such that with these initial conditions, we
have $\varepsilon(t)=0$ and $\lambda(t)=\lambda$}%
\label{airplane.fig}%
\end{figure}

We then make use of the second Hamilton--Jacobi formula (\ref{derivFormulaGen}%
), which allows us to compute the derivatives of \thinspace$S$ directly,
without having to attempt to differentiate the formula (\ref{SmultFormula})
for $S.$ Working in logarithmic polar coordinates, $\rho=\log\left\vert
\lambda\right\vert $ and $\theta=\arg\lambda,$ we find an amazingly simple
expression for the quantity
\[
\frac{\partial s_{t}}{\partial\rho}=\lim_{\varepsilon\rightarrow0^{+}}%
\frac{\partial S}{\partial\rho}(t,\lambda,\varepsilon),
\]
inside $\Sigma_{t},$ namely,%
\begin{equation}
\frac{\partial s_{t}}{\partial\rho}=\frac{2\rho}{t}+1,\quad\lambda\in
\Sigma_{t}.\label{dsdRho}%
\end{equation}
(See Corollary 7.6 in \cite{DHKBrown}.) This result is obtained using a
certain constant of motion of the system of ODE's, namely the quantity%
\[
\Psi=\varepsilon p_{\varepsilon}+\frac{1}{2}(ap_{a}+bp_{b})
\]
in \cite[Proposition 6.5]{DHKBrown}.

If we evaluate this constant of motion at a time $t$ when $\varepsilon(t)=0,$
the $\varepsilon p_{\varepsilon}$ term vanishes. But if $\varepsilon(t)=0,$
the second Hamilton--Jacobi formula (\ref{derivFormulaGen}) tells us that%
\[
\left(  a\frac{\partial S}{\partial a}+b\frac{\partial S}{\partial b}\right)
(t,\lambda(t),0)=a(t)p_{a}(t)+b(t)p_{b}(t).
\]
Furthermore, $a\frac{\partial S}{\partial a}+b\frac{\partial S}{\partial b}$
is just $\partial S/\partial\rho,$ computed in rectangular coordinates. A bit
of algebraic manipulation yields an explicit formula for $a\frac{\partial
S}{\partial a}+b\frac{\partial S}{\partial b},$ as in \cite[Theorem
6.7]{DHKBrown}, explaining the formula (\ref{dsdRho}). To complete the proof
(\ref{dsdRho}), it still remains to address certain regularity issues of
$S(t,\lambda,\varepsilon)$ near $\varepsilon>0,$ as in Section 7.3 of
\cite{DHKBrown}.

Once (\ref{dsdRho}) is established, we note that the formula for $\partial
s_{t}/\partial\rho$ in (\ref{dsdRho}) is independent of $\theta.$ It follows
that%
\[
\frac{\partial}{\partial\rho}\frac{\partial s_{t}}{\partial\theta}%
=\frac{\partial}{\partial\theta}\frac{\partial s_{t}}{\partial\rho}=0,
\]
that is, that $\partial s_{t}/\partial\theta$ is independent of $\rho$ inside
$\Sigma_{t}.$ Writing the Laplacian in logarithmic polar coordinates, we then
find that%
\begin{align}
\Delta s_{t}(\lambda)  &  =\frac{1}{r^{2}}\left(  \frac{\partial^{2}s_{t}%
}{\partial\rho^{2}}+\frac{\partial^{2}s_{t}}{\partial\theta^{2}}\right)
\nonumber\\
&  =\frac{1}{r^{2}}\left(  \frac{2}{t}+\frac{\partial}{\partial\theta}\left(
\frac{\partial s_{t}}{\partial\theta}\right)  \right)  ,\quad\lambda\in
\Sigma_{t}, \label{secondDeriv}%
\end{align}
where $2/t$ term in the expression comes from differentiating (\ref{dsdRho})
with respect to $\rho.$ Since $\partial s_{t}/\partial\theta$ is independent
of $\rho,$ we can understand the structure of the formula in Theorem
\ref{dhkMain.thm}.

The last step in the proof of Theorem \ref{dhkMain.thm} is to compute
$\partial s_{t}/\partial\theta.$ Since $\partial s_{t}/\partial\theta$ is
independent of $\rho$---or, equivalently, independent of $r=\left\vert
\lambda\right\vert $---inside $\Sigma_{t},$ the value of $\partial
s_{t}/\partial\theta$ at a point $\lambda$ in $\Sigma_{t}$ is the same as its
value as we approach the boundary of $\Sigma_{t}$ along the radial segment
through $\lambda.$ We show that $\partial s_{t}/\partial\theta$ is continuous
over the whole complex plane, even at the boundary of $\Sigma_{t}.$ (See
Section 7.4 of \cite{DHKBrown}.) Thus, on the boundary of $\Sigma_{t},$ the
function $\partial s_{t}/\partial\theta$ will agree with the angular
derivative of $\log(\left\vert \lambda-1\right\vert ^{2})$, namely%
\begin{align}
\frac{\partial}{\partial\theta}\log(\left\vert \lambda-1\right\vert ^{2})  &
=\frac{2\operatorname{Im}\lambda}{\left\vert \lambda-1\right\vert ^{2}%
}\nonumber\\
&  =\frac{2r\sin\theta}{r^{2}+1-2r\cos\theta}. \label{dSdthetaOut}%
\end{align}
Thus, to compute $\partial s_{t}/\partial\theta$ at a point $\lambda$ in
$\Sigma_{t},$ we simply evaluate (\ref{dSdthetaOut}) at either of the two
points where the radial segment through $\lambda$ intersects $\partial
\Sigma_{t}.$ (We get the same value at either point.)

One such boundary point is the point with argument $\theta=\arg\lambda$ and
radius $r_{t}(\theta),$ as in Figure \ref{rt.fig}. Thus, inside $\Sigma_{t},$
we have%
\[
\frac{\partial s_{t}}{\partial\theta}=\frac{2r_{t}(\theta)\sin\theta}%
{r_{t}(\theta)^{2}+1-2r_{t}(\theta)\cos\theta}.
\]
Plugging this expression into (\ref{secondDeriv}) gives the claimed formula in
Theorem \ref{dhkMain.thm}.


\begin{thebibliography}{99}                                                                                               %


\bibitem {Bai}Z. D. Bai, Circular law, \textit{Ann. Probab.} \textbf{25}
(1997), 494--529.

\bibitem {BianeConvolution}P. Biane, On the free convolution with a
semi-circular distribution, \textit{Indiana Univ. Math. J.} \textbf{46}
(1997), 705--718.

\bibitem {BianeFields}P. Biane, Free Brownian motion, free stochastic calculus
and random matrices.\ \textit{In}\ Free Probability Theory (Waterloo, ON,
1995), 1--19. Fields Institute Communications 12. Providence, RI: American
Mathematical Society, 1997.

\bibitem {BianeJFA}P. Biane, Segal--Bargmann transform, functional calculus on
matrix spaces and the theory of semi-circular and circular systems, \textit{J.
Funct. Anal.} \textbf{144} (1997), 232--286.

\bibitem {BS1}P. Biane and R. Speicher, Stochastic calculus with respect to
free Brownian motion and analysis on Wigner space, \textit{Probab. Theory
Related Fields} \textbf{112} (1998), 373--409.

\bibitem {BK}P. Bourgade and J. P. Keating, Quantum chaos, random matrix
theory, and the Riemann $\zeta$-function. \textit{In }Chaos, 125--168, Prog.
Math. Phys., 66, Birkh\"{a}user/Springer, Basel, 2013.

\bibitem {Br}Brown, L. G. Lidski\u{\i}'s theorem in the type II case.
\textit{In} Geometric methods in operator algebras (Kyoto, 1983), 1--35,
Pitman Res. Notes Math. Ser., 123, Longman Sci. Tech., Harlow, 1986.

\bibitem {Ceb}G. C\'{e}bron, Free convolution operators and free Hall
transform, \textit{J. Funct. Anal.} \textbf{265} (2013), 2645--2708.

\bibitem {DHKLargeN}B. K. Driver, B. C. Hall, and T. Kemp, The large-$N$ limit
of the Segal--Bargmann transform on $\mathbb{U}_{N}$, \textit{J. Funct. Anal.}
\textbf{265} (2013), 2585--2644.

\bibitem {DHKBrown}B. K. Driver, B. C. Hall, and T. Kemp, The Brown measure of
the free multiplicative Brownian motion, preprint arXiv:1903.11015 [math.PR].

\bibitem {Evans}L. C. Evans, Partial differential equations. Second edition.
Graduate Studies in Mathematics, 19. American Mathematical Society,
Providence, RI, 2010. xxii+749 pp.

\bibitem {FPZ}O. Feldheim, E. Paquette, and O. Zeitouni, Regularization of
non-normal matrices by Gaussian noise, \textit{Int. Math. Res. Not. IMRN}
\textbf{18} (2015), 8724--8751.

\bibitem {FK1}B. Fuglede and R. V. Kadison, On determinants and a property of
the trace in finite factors, \textit{Proc. Nat. Acad. Sci. U. S. A.}
\textbf{37} (1951), 425--431.

\bibitem {FK2}B. Fuglede and R. V. Kadison, Determinant theory in finite
factors, \textit{Ann. of Math. (2)} \textbf{55} (1952), 520--530.

\bibitem {Gin}J. Ginibre, Statistical ensembles of complex, quaternion, and
real matrices, \textit{J. Math. Phys.} \textbf{6} (1965), 440--449.

\bibitem {Girko}V. L. Girko, The circular law. (Russian) \textit{Teor.
Veroyatnost. i Primenen.} \textbf{29} (1984), 669--679.

\bibitem {GM}L. Gross and P. Malliavin, Hall's transform and the
Segal--Bargmann map. \textit{In} It\^{o}'s stochastic calculus and probability
theory (N. Ikeda, S. Watanabe, M. Fukushima and H. Kunita, Eds.), 73--116,
Springer, 1996.

\bibitem {Nowak}E. Gudowska-Nowak, R. A. Janik, J. Jurkiewicz, and M. A.
Nowak, Infinite products of large random matrices and matrix-valued diffusion,
\textit{Nuclear Phys. B} \textbf{670} (2003), 479--507.

\bibitem {GWZ}A. Guionnet, P. M. Wood, and O. Zeitouni, Convergence of the
spectral measure of non-normal matrices, \textit{Proc. Amer. Math. Soc.}
\textbf{142} (2014), 667--679.

\bibitem {Gutz}M. C. Gutzwiller, Chaos in classical and quantum mechanics.
Interdisciplinary Applied Mathematics, \textbf{1}. Springer-Verlag, New York, 1990.

\bibitem {Ha1994}B. C. Hall, The Segal--Bargmann \textquotedblleft coherent
state\textquotedblright\ transform for compact Lie groups. \textit{J. Funct.
Anal.} \textbf{122} (1994), 103--151.

\bibitem {Hall2001}B. C. Hall, Harmonic analysis with respect to heat kernel
measure, \textit{Bull. Amer. Math. Soc. (N.S.)} \textbf{38} (2001), 43--78.

\bibitem {HallQM}B. C. Hall, Quantum theory for mathematicians. Graduate Texts
in Mathematics, \textbf{267}. Springer, New York, 2013.

\bibitem {HallLie}B. C. Hall, Lie groups, Lie algebras, and representations.
An elementary introduction. Second edition. Graduate Texts in Mathematics,
\textbf{222}. Springer, 2015.

\bibitem {HallExpository}B. C. Hall, The Segal--Bargmann transform for unitary
groups in the large-$N$ limit, preprint arXiv:1308.0615 [math.RT].

\bibitem {HK}B. C. Hall and T. Kemp, Brown measure support and the free
multiplicative Brownian motion, \textit{Adv. Math.} \textbf{355} (2019),
article 106771, 36 pp.

\bibitem {Higham}N. J. Higham, Functions of matrices. Theory and computation.
Society for Industrial and Applied Mathematics (SIAM), Philadelphia, PA, 2008.

\bibitem {Ho}C.-W. Ho, The two-parameter free unitary Segal-Bargmann transform
and its Biane-Gross-Malliavin identification, \textit{J. Funct. Anal.}
\textbf{271} (2016), 3765--3817.

\bibitem {HZ}C.-W. Ho and P. Zhong, Brown Measures of free circular and
multiplicative Brownian motions with probabilistic initial point, preprint
arXiv:1908.08150 [math.OA].

\bibitem {KS}N. M. Katz and P. Sarnak, Zeroes of zeta functions and symmetry,
\textit{Bull. Amer. Math. Soc. (N.S.)} \textbf{36} (1999), 1--26.

\bibitem {KempLargeN}T. Kemp, The large-$N$ limits of Brownian motions on
$\mathbb{GL}_{N}$, \textit{Int. Math. Res. Not.}, (2016), 4012--4057.

\bibitem {Lohmayer}R. Lohmayer, H. Neuberger, and T. Wettig, Possible
large-$N$ transitions for complex Wilson loop matrices, \textit{J. High Energy
Phys.} 2008, no. 11, 053, 44 pp.

\bibitem {Mehta}M. L. Mehta, Random matrices. Third edition. Pure and Applied
Mathematics (Amsterdam), 142. Elsevier/Academic Press, Amsterdam, 2004

\bibitem {MS}J. A. Mingo and R. Speicher, Free probability and random
matrices. Fields Institute Monographs, 35. Springer, New York; Fields
Institute for Research in Mathematical Sciences, Toronto, ON, 2017.

\bibitem {Mon}H. L. Montgomery, The pair correlation of zeros of the zeta
function. Analytic number theory (Proc. Sympos. Pure Math., Vol. XXIV, St.
Louis Univ., St. Louis, Mo., 1972), pp. 181--193. Amer. Math. Soc.,
Providence, R.I., 1973.

\bibitem {NS}A. Nica and R. Speicher, Lectures on the combinatorics of free
probability. London Mathematical Society Lecture Note Series, 335. Cambridge
University Press, Cambridge, 2006.

\bibitem {blueRudin}W. Rudin, Principles of mathematical analysis. Third
edition. International Series in Pure and Applied Mathematics. McGraw-Hill
Book Co., New York-Auckland-D\"{u}sseldorf, 1976.

\bibitem {Sniady}P. \'{S}niady, Random regularization of Brown spectral
measure, \textit{J. Funct. Anal.} \textbf{193} (2002), 291--313.

\bibitem {Stockmann}H.-J. St\"{o}ckmann, Quantum chaos. An introduction.
Cambridge University Press, Cambridge, 1999.

\bibitem {Tao}T. Tao, Topics in random matrix theory. Graduate Studies in
Mathematics, 132. American Mathematical Society, Providence, RI, 2012.

\bibitem {Voi1}D. Voiculescu, Symmetries of some reduced free product
$C^{\ast}$-algebras. \textit{In} \textquotedblleft Operator algebras and their
connections with topology and ergodic theory (Bu\c{s}teni,
1983),\textquotedblright\ 556--588, Lecture Notes in Math., 1132, Springer,
Berlin, 1985.

\bibitem {Voi2}D. Voiculescu, Limit laws for random matrices and free
products, \textit{Invent. Math.} \textbf{104} (1991), 201--220.

\bibitem {Wigner}E. Wigner, Characteristic vectors of bordered matrices with
infinite dimensions. \textit{Ann. of Math. (2)} \textbf{62} (1955), 548--564.
\end{thebibliography}
\end{document}